\documentclass[a4paper,11pt]{article}
\usepackage{epsf,epsfig}
\usepackage{graphics,color}
\usepackage{amsmath}
\usepackage{amssymb}
\usepackage{cite}
\usepackage{verbatim}
\usepackage{float}
\usepackage{graphicx}
\usepackage{amsthm}
\usepackage{textcomp}
\usepackage{subfig}
\usepackage{enumitem}
\usepackage{bigints}

\newtheorem{theorem}{Theorem}[section]
\newtheorem{prop}[theorem]{Proposition}
\newtheorem{cor}[theorem]{Corollary}
\newtheorem{lemma}[theorem]{Lemma}

\newtheorem{remark}[theorem]{Remark}

\numberwithin{equation}{section}
\definecolor{Aquamarine}{cmyk}{0.82,0,0.30,0}

\newcommand{\black}{\color{black}}

\newcommand{\jump}[1]{\text{{\rm \textlbrackdbl}}{#1}\text{{\rm \textrbrackdbl}}}

\def\altrolifting{\omega}
\def\ambientsource{\Rn}

\def\apprextliftk{\widehat \varphi_k}
\def\apprextuk{\widehat u_k}
\def\ATeps{{\rm AT}_\var}
\def\ATepsk{{\rm AT}_{\var_k}}

\def\AToneeps{{\widehat{\rm AT}_\var^{\Suno}}}
\def\AToneepsk{{\rm AT}^{\Suno}_{\var_k}}

\def\ATtwoeps{{{\rm AT}_\var^{\Suno}}}
\def\ATtwoepsk{{{\rm AT}_{\var_k}^{\Suno}}}

\def\circle{{\mathbb S}^1}
\def\codone{\onecod}
\def\constlemmabase{C}
\def\constlemmasuccessivo{C}

\def\dimension{n}
\def\dki{d_k^{(i)}}
\def\dist{\,{\rm dist\ }}
\def\domAT{\mathcal A}
\def\domATgrande{\widehat{\mathcal A}_{\Suno}}

\def\domATpiccolo{\domAT_{\Suno}}
\def\etaprimo{\eta'}
\def\etasecondo{\eta''}
\def\eps{\epsilon}
\def\extlift{\widehat \varphi}
\def\extu{{\widehat u}}
\def\extuk{\widehat u_k}
\def\finperset{G}
\def\funzioni_vito_k{\phi_k}

\def\grad{\nabla}

\def\Hausdorff{\mathcal H^{\dimension-1}}

\def\indice{j}
\def\indiceteo{m}
\def\limk{\lim_{k \to +\infty}}
\def\liminfk{\liminf_{k \to +\infty}}
\def\limsupk{\limsup_{k \to +\infty}}
\def\liftexample{\varphi}
\def\livellok{t(k)}
\def\Minkowski{\mathcal M^{\dimension-1}}
\def\minvalueup{m_p[u]}
\def\minvalueu2{m_2[u]}
\def\MMepsk{{\rm MM}_{\var_k}}
\def\ModicaMortolaeps{{\rm MM}_\var}
\def\ModicaMortolaepsk{{\rm MM}_{\var_k}}
\def\ModicaMortola{{\rm MM}}
\def\MS{{\rm MS}}
\def\MSlift{{\rm MS}_{{\rm lift}}}
\def\MSSuno{{\rm MS}_{\circle}}
\def\NN{{\mathbb N}}
\def\nuovoliftingk{\phi_k}
\newcommand{\Om}{\Omega}
\def\onecod{{\dimension-1}}
\def\openAT{\Omega}
\def\opensetAT{\openAT}
\def\opensetcompactness{\Om}
\def\openset{U}
\newcommand{\pip}{(i)}
\newcommand{\pNp}{(N)}
\newcommand{\pNpop}{(N+1)}
\newcommand{\R}{\mathbb{R}}
\def\Rn{{\mathbb R}^\dimension}
\def\sopraaperto{\widehat \Om}
\def\sottoaperto{A}
\def\suk{(u_k)}
\def\Suno{\circle}
\def\supp{\,{\rm supp \ }}
\def\trunkphi{\varphi}
\def\var{\varepsilon}

\newcommand{\dx}{\, {\rm d}x}
\newcommand{\dy}{\, {\rm d}y}

\newcommand{\ds}{\, {\rm d}s}
\newcommand{\dt}{\, {\rm d}t}

\newcommand{\res}      {\mathop{\hbox{\vrule height 7pt width .5pt depth
			0pt\vrule height .5pt width 6pt depth 0pt}}\nolimits}

\title{On jump minimizing liftings for $\mathbb{S}^1$-valued maps 
and connections with Ambrosio-Tortorelli-type $\Gamma$-limits
}
\author{
	Giovanni Bellettini\footnote{
Dipartimento di Ingegneria dell'Informazione e Scienze Matematiche, Universit\`a di Siena, 53100 Siena, Italy,
and International Centre for Theoretical Physics ICTP,
Mathematics Section, 34151 Trieste, Italy.
	E-mail: bellettini@unisi.it
}
\and
	Roberta Marziani\footnote{Dipartimento di Ingegneria dell'Informazione e Scienze Matematiche, Universit\`a di Siena, 53100 Siena, Italy.
	E-mail: roberta.marziani@unisi.it}
	\and
	Riccardo Scala\footnote{Dipartimento di Ingegneria dell'Informazione e Scienze Matematiche, Universit\`a di Siena, 53100 Siena, Italy.
		E-mail: riccardo.scala@unisi.it}
}

\begin{document}
	\maketitle
	
	\begin{abstract}
This paper is concerned with the $\Gamma$-limits
of Ambrosio-Tortorelli-type functionals, for maps
$u$ defined on an open bounded set $\Om \subset \Rn$
and taking values in the unit circle $\mathbb S^1 \subset \R^2$. Depending 
on the domain of the functional, 
two different $\Gamma$-limits are possible, 
one of which is nonlocal, and related to 
the notion of jump minimizing lifting,
i.e., a lifting of a map $u$  whose measure of the jump set
is minimal. {The latter requires ad hoc compactness results for sequences of liftings which, besides being interesting by themselves, also allow to deduce existence of a jump minimizing lifting. }
	\end{abstract}
	

	\vspace{2mm}
	
	\noindent {\bf AMS (MOS) subject clas\-si\-fi\-ca\-tion:} 
	49Q15, 49Q20, 49J45, 58E12.
	
	\section{Introduction}

This paper is devoted to the asymptotic analysis, 
via $\Gamma$-convergence, of regularized free discontinuity 
functionals for maps defined on a connected bounded open set
$\Om \subset \Rn$ 
with Lipschitz boundary
and taking values in the unit circle $\mathbb{S}^1 \subset \R^2$. 
More specifically, 
define the two functional domains

\begin{equation}\label{eq:two_domains}
	\begin{aligned}
		\domATgrande
:=&\left\{(u,v)\in W^{1,1}(\Omega;\mathbb{S}^1)\times W^{1,2}(\Omega)\colon v|\nabla u|\in L^2(\Om),\ 0\le v\le1
		\right\} 
\subset L^1(\Om; \circle) \times L^1(\Om),
		\\
		\domATpiccolo
:=&\left\{(u,v)\in W^{1,2}(\Omega;\mathbb{S}^1)\times W^{1,2}(\Omega)\colon \ 0\le v\le1
		\right\}\subset 
L^1(\Om; \circle) \times L^1(\Om).
	\end{aligned}
\end{equation}
For $\var \in (0,1]$
let us consider the corresponding family of 
Ambrosio-Tortorelli-type functionals
$$\AToneeps, 
\ATtwoeps\colon L^1(\Omega;\mathbb{S}^1)\times L^1(\Omega)\to[0,+\infty]$$ 
given by
\begin{equation}\label{AT1}
	\AToneeps(u,v):=\begin{cases}
		\ATeps(u,v)& \text{ if } (u,v)\in\domATgrande,
\\[1em]
		+\infty&\text{ otherwise 
in } L^1(\Om; \circle) \times L^1(\Om),
	\end{cases}
\end{equation}
and
\begin{equation}\label{AT2}
	\ATtwoeps(u,v):=\begin{cases}
		\ATeps(u,v)& \text{ if } (u,v)\in\domATpiccolo,
\\[1em]
		+\infty&\text{ otherwise in }
L^1(\Om; \circle) \times L^1(\Om),
	\end{cases}
\end{equation}
where
\begin{equation}\label{AT}
	\ATeps(u,v):=
	\int_{\Omega}\left(v^2|\nabla u|^2+\varepsilon|\nabla v|^2+\frac{(v-1)^2}{4\varepsilon}\right){\rm d}x,
\end{equation}
with $\vert \grad u\vert^2$ indicating the Frobenius norm of $\grad u$.
It is clear that
$$
\domATpiccolo
 \subset 		\domATgrande, \qquad 
\ATtwoeps \geq 
\AToneeps.
$$

Our main results, connecting these functionals
with Mumford-Shah-type functionals, read as follows.

\begin{theorem}[$\Gamma$-convergence of $\AToneeps$]\label{thm1}
	Let $\Omega\subseteq\R^n$ be a connected bounded open set with Lipschitz boundary.
We have 
 $$
\Gamma-L^1 \lim_{\eps \to 0^+} \AToneeps = \MSSuno,
$$
 where
$\MSSuno\colon L^1(\Omega;\mathbb{S}^1)\times L^1(\Omega)\to[0,+\infty]$ is 
given by
	\begin{equation}\label{MS}
		\MSSuno(u,v):=
		\begin{cases}
			\displaystyle	\int_{\Omega}|\nabla u|^2 {\rm d}x+ \Hausdorff(S_u)
			& {\rm if}~ u\in SBV^2(\Omega;\mathbb{S}^1),\, v=1 \text{ a.e.}\,,
			\\[1em]
			+\infty & {\rm otherwise ~in} ~ L^1(\Om; \Suno) \times L^1(\Om).
		\end{cases}
	\end{equation}
\end{theorem}
\begin{theorem}[$\Gamma$-convergence of $\ATtwoeps$]\label{thm2}
{ Let $\Omega\subseteq\R^n$ be a connected and  simply-connected bounded open set with Lipschitz boundary.} We have 
 $$
\Gamma-L^1 \lim_{\eps \to 0^+} \ATtwoeps = \MSlift,
$$
 where
$\MSlift
\colon L^1(\Omega;\mathbb{S}^1)\times L^1(\Omega)\to[0,+\infty]$ is given by
	\begin{equation}\label{MSlif}
		\MSlift(u,v):=
		\begin{cases}
			\displaystyle	\int_\Omega|\nabla u|^2\dx+
\minvalueu2
 & 
			{\rm if~} u\in SBV^2(\Omega;\mathbb{S}^1),\, v=1 \text{ a.e.} 
			\\[1em]
			+\infty &{\rm otherwise~in~}
L^1(\Om; \circle) \times L^1(\Om),
		\end{cases}
	\end{equation}
	with 
\begin{equation}\label{eq:minprob_intro}
\minvalueu2:=\inf\{
\Hausdorff(S_\varphi)
\colon\varphi\in GSBV^2(\Omega),\, e^{i\varphi}=u\text{ a.e. in }\Omega\}.
\end{equation}
\end{theorem}
In \eqref{MS} and
\eqref{eq:minprob_intro} the symbol
 $\mathcal{H}^{n-1}(S_\varphi)$
stands for the $(\codone)$-dimensional Hausdorff measure of the jump
set $S_\varphi$ of $\varphi$, a lifting\footnote{See Section 
\ref{subsec:liftings_of_S1_maps}.} of $u$,
and $SBV^2(\Om; \Suno)$ (resp. $GSBV^2(\Om)$)
is the space of $\Suno$-valued maps with 
special bounded variation in $\Om$ (resp.
the space of generalized
special bounded variation functions in $\Om$)
whose absolutely
continuous part of the gradient is square integrable and whose jump set has finite $(n-1)$-dimensional Hausdorff measure.
To better understand the above results, some
comments are in order.

\begin{itemize}
\item[(i)] In Theorem \ref{thm1} 
the $\Gamma$-limit in \eqref{MS} is the classical Mumford-Shah functional
for $\mathbb{S}^1$-valued maps, and indeed part of 
the proof is an adaptation
of known results, together with some applications of the properties
of liftings of nonsmooth maps;
\item[(ii)]
the second, more interesting, 
$\Gamma$-limit in Theorem \ref{thm2} is nonlocal, it
does not have an integral representation and depends on the minimization problem in \eqref{min-prob}.
The singular term $m_2[u]$ is the penalization of 
a particular lifting of $u$, the one which minimizes the measure
of the jump. 
It is clear that
$\MSlift \geq \MSSuno$.
\end{itemize}

These two theorems show the huge difference made by the choice of the 
two domains \eqref{eq:two_domains}
where defining the approximating functionals,
in particular the requirements $u \in W^{1,1}(\Om; \Suno)$,
$v\vert \grad u\vert \in L^2(\Om)$, as opposite to the more standard
requirement $u \in W^{1,2}(\Om; \Suno)$.
The reason of the corresponding different limit
behaviours is topological in nature, and
it is better explained by the following example, 
see also \cite[pag. 30]{DLSVG}.  \\

\noindent
{\bf The example of the vortex map.} Let $n=2$, $\Om = B_1(0) \subset \R^2$,  and consider the vortex map
$u_V(x) = \frac{x}{\vert x\vert}$ for any $x \in B_1(0)\setminus \{0\}$. 
Then $u_V \in 
W^{1,p}(B_1(0); \Suno)$ for any $p \in [1,2)$ and $u 
\notin W^{1,2}(B_1(0); \Suno)$. 
Thus in particular
$$
\MSSuno(u_V, 1) = +\infty \qquad 
{ \text{since } u_V\notin SBV^2(\Omega;\mathbb{S}^1).}
$$
Also, any lifting of $u_V$ jumps (in $B_1(0)$) at least on some curve connecting the origin to $\partial B_1(0)$. 
{Now, let us modify $u_V$ in a small neighbourhood of the origin to get a function which is in $SBV^2(\Omega;\mathbb S^1)$, by introducing a small jump. To this purpose} consider the lifting of $u_V$ given by 
the argument function\footnote{I.e., the imaginary part
of the complex logarithm.} $\theta$
jumping on the 
positive 
real axis. {Hence, in $\Omega$, $S_\theta=(0,1)\times\{0\}$ with jump opening $\jump{\theta}=2\pi$.} 
Now, for $0 < r <R \leq 1$ define the 
annulus $A_{r, R} = B_R(0) \setminus \overline B_r(0)$,
and let $\sigma \in (0,1)$. Consider the 
following perturbation of $\theta$ and $u_V$
on $B_1(0) \setminus \{0\}$: let $\chi_\sigma 
\in C^\infty(B_\sigma(0), 
[0,1])$ be such that $\chi_\sigma \equiv 0$ in $B_{\sigma/4}(0)$, 
$\chi_\sigma \equiv 1$ in $A_{3/4 \sigma, \sigma}$,   $0<\chi_\sigma<1$ in $A_{\sigma/4,3\sigma/4}$,
$\vert \grad \chi_\sigma\vert 
\leq \frac{C}{\sigma}$ for some $C>0$, 
$$
\theta_\sigma(x) :=
\begin{cases}
\chi_\sigma(x) \theta(x) & {\rm if}~ x\in B_\sigma(0) \setminus \{0\},
\\
\theta(x) & {\rm if}~ x \in A_{\sigma,1}(0),
\end{cases}
\qquad \qquad
u^{(\sigma)} := e^{i\theta_\sigma}. 
$$
 \begin{figure}\centering
 	\def\svgwidth{0.40\textwidth}
\begingroup%
  \makeatletter%
  \providecommand\color[2][]{%
    \errmessage{(Inkscape) Color is used for the text in Inkscape, but the package 'color.sty' is not loaded}%
    \renewcommand\color[2][]{}%
  }%
  \providecommand\transparent[1]{%
    \errmessage{(Inkscape) Transparency is used (non-zero) for the text in Inkscape, but the package 'transparent.sty' is not loaded}%
    \renewcommand\transparent[1]{}%
  }%
  \providecommand\rotatebox[2]{#2}%
  \newcommand*\fsize{\dimexpr\f@size pt\relax}%
  \newcommand*\lineheight[1]{\fontsize{\fsize}{#1\fsize}\selectfont}%
  \ifx\svgwidth\undefined%
    \setlength{\unitlength}{511.466992bp}%
    \ifx\svgscale\undefined%
      \relax%
    \else%
      \setlength{\unitlength}{\unitlength * \real{\svgscale}}%
    \fi%
  \else%
    \setlength{\unitlength}{\svgwidth}%
  \fi%
  \global\let\svgwidth\undefined%
  \global\let\svgscale\undefined%
  \makeatother%
  \begin{picture}(1,0.99999992)%
    \lineheight{1}%
    \setlength\tabcolsep{0pt}%
    \put(0,0){\includegraphics[width=\unitlength,page=1]{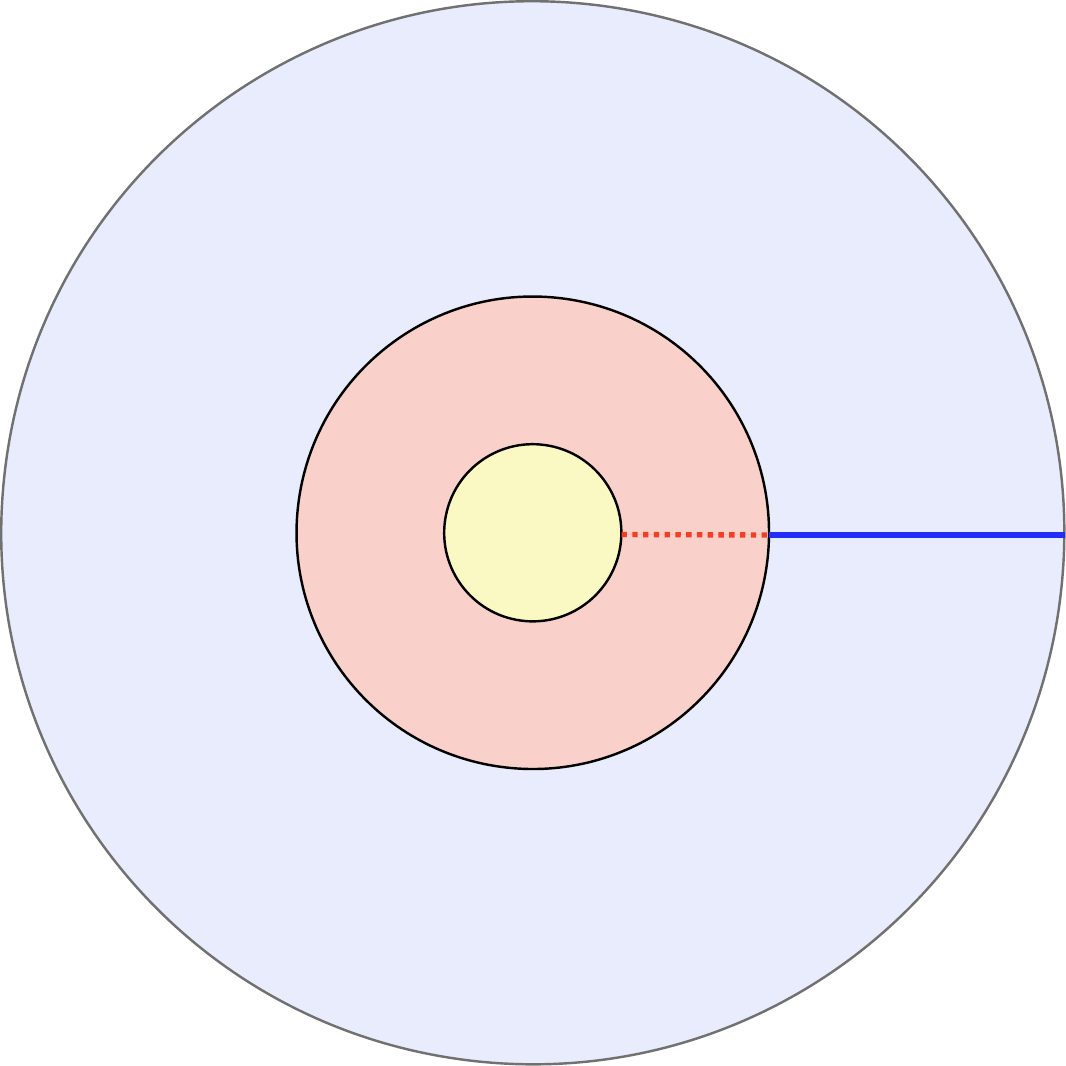}}%
     \put(0.72005078,0.77597175){\color[rgb]{0,0,0}\makebox(0,0)[lt]{\lineheight{1.25}\smash{\begin{tabular}[t]{l}$A_{\frac{3\sigma}4,1}$\end{tabular}}}}%
       \put(0.54005078,0.58597175){\color[rgb]{0,0,0}\makebox(0,0)[lt]{\lineheight{1.25}\smash{\begin{tabular}[t]{l}$A_{\frac\sigma4,\frac{3\sigma}4}$\end{tabular}}}}%
         \put(0.41005078,0.48597175){\color[rgb]{0,0,0}\makebox(0,0)[lt]{\lineheight{1.25}\smash{\begin{tabular}[t]{l}$B_{\frac\sigma4}(0)$\end{tabular}}}}%
    \put(0.79005078,0.42597175){\color[rgb]{0,0,0}\makebox(0,0)[lt]{\lineheight{1.25}\smash{\begin{tabular}[t]{l}$S^I_{\theta_\sigma}$\end{tabular}}}}%
    \put(0.58205078,0.42597175){\color[rgb]{0,0,0}\makebox(0,0)[lt]{\lineheight{1.25}\smash{\begin{tabular}[t]{l}$S^f_{\theta_\sigma}$\end{tabular}}}}%
  \end{picture}%
\endgroup%

 	\caption{The dotted segment denotes the set $S^f_{\theta_\sigma}$ where $\jump{\theta_\sigma}$ varies between 0 and $2\pi$, the continuous segment  denotes the set $S^I_{\theta_\sigma}$ where $\jump{\theta_\sigma}=2\pi$. }\label{fig:example}
 \end{figure}
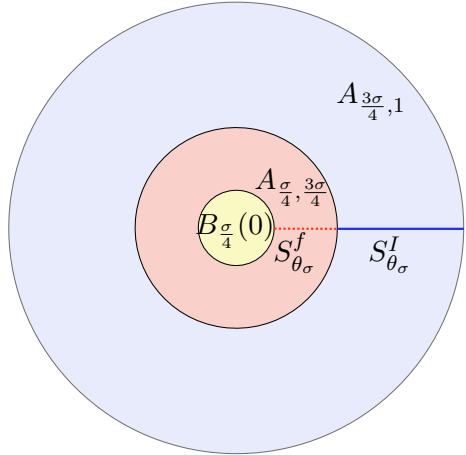
{Thus we have (see Figure \ref{fig:example})
\begin{equation*}
S_{\theta_\sigma}=S^f_{\theta_\sigma}\cup S^I_{\theta_\sigma}\quad \text{where}\quad S^f_{\theta_\sigma}=\Big(\frac\sigma4,\frac{3\sigma}4\Big)\times\{0\}\,,\quad 
S^I_{\theta_\sigma}=\Big(\frac{3\sigma}4,1\Big)\times\{0\}\,,
\end{equation*}
moreover $\jump{\theta_\sigma}$ varies from $0$ to $2\pi$ on $S^f_{\theta_\sigma}$ and  $\jump{\theta_\sigma}\equiv2\pi$ on  $S^I_{\theta_\sigma}$. This in turn implies $u^{(\sigma)}=u_V$ on $A_{3\sigma/4,1}$ and $S_{u^{(\sigma)}}=S^f_{\theta_\sigma}$.
}
We also have $u^{(\sigma)}
\in SBV^2(B_1(0); \Suno)$, and $\MSSuno(u^{(\sigma)},1) < +\infty$,
and 
$m_2[u^{(\sigma)}]$ turns out to be
the length of the shortest segment joining the right extremum of $S_{u^{(\sigma)}}$ 
to the boundary of $\Omega$. {Indeed any lifting $\varphi$ of ${u^{(\sigma)}}$ has the following properties: $S_\varphi\supset S_{\theta_\sigma}^f$, $\varphi$ is a lifting of $u_V$ on  $A_{3\sigma/4,1}$ and  $\varphi$ is a lifting of $(1,0)$ on $B_\sigma(0)$. Hence, in order to be minimal it must be $S_\varphi= S_{\theta_\sigma}$ and therefore $m_2[u^{(\sigma)}]=\mathcal{H}^{1}(S_{\theta_\sigma})=1-\sigma/4$.}

The map $u^{(\sigma)}$ is one of the possible maps
on which one may compute the two $\Gamma$-limits above.
It turns out that the 
gradient integrability of the approximating
maps is crucial to devise the two behaviours.
In particular 
in Theorem \ref{thm2}, due to topological obstructions related to
{the fact that $u^{(\sigma)}$ has non-zero degree {on any circle $\partial B_\rho(0)$ with $ 3\sigma/4\leq \rho< 1$} (as it coincides with the vortex map outside a small neighbourhood of the origin)}, the recovery sequence gives rise to 
the contribution 
$m_2[u^{(\sigma)}]$. {Indeed, a naive approach to get the upper bound would be to adapt the classical construction given in \cite{AT90,AT92}. More precisely, the latter consists in regularising $u^{(\sigma)}$ in a neighbourhood  with width $\ll\var$ of its jump set $S_{u^{(\sigma)}}$, 
being careful to keep values on the unit circle $\mathbb{S}^1$, and then define $v_\var$ accordingly.
 However this leads to an approximating function $u_\var\colon B_1(0)\to\mathbb{S}^1$ which coincides with the vortex map far from the origin, and so necessarily  $u_\var$ has non-zero degree, which in turn implies\footnote{Indeed, if $u_\var\in W^{1,2}(B_1(0),\mathbb{S}^1)$, then by \cite[Theorem 1.1]{Brezis-Mironescu} $u_\var$ has a lifting in $W^{1,2}(B_1(0))$, which is not possible since $u_\var$ has non zero degree (cf. \cite[Theorem 1.2]{Brezis-Mironescu}).}  that $u_\var \notin W^{1,2}(B_1(0),\mathbb{S}^1)$.  In particular, $u_\var$ cannot be used to produce a recovery sequence in Theorem \ref{thm2}, but  it
 provides a suitable construction for the upper bound in Theorem \ref{thm1},  being of class $W^{1,1}(B_1(0);\mathbb{S}^1)$ and satisfying $v_\var|\nabla u_\var|\in L^2(B_1(0))$.
In order to construct a recovery sequence in the second case it is needed to consider  approximating functions with zero degree. This is possible by modifying $u^{(\sigma)}$ in a neighbourhood of the jump set of a minimal lifting (for example $\theta_\sigma$) and then define $v_\var$ accordingly.  Thus, passing to the limit one gets the contribution  $\mathcal{H}^1(S_{\theta_\sigma})=m_2[u^{(\sigma)}]$.
}\\

\noindent
{\bf Main challenges of the proofs.} {From the previous example we learned that the nature of the $\Gamma$-limit \eqref{MSlif} is related to topological obstructions when approximating a $\mathbb{S}^1$-valued map whose approximate gradient is square integrable. In particular, the non locality of ${\rm MS}_{\rm lift}$ leads to non trivial difficulties in the proof of the lower bound inequality. The  basic idea consists in rewriting the approximating functionals in terms of liftings. Namely, given $u\in W^{1,2}(\Omega;\mathbb{S}^1)$ take a lifting $\varphi\in W^{1,2}(\Omega)$ of $u$, so that from the identity $|\nabla u|=|\nabla\varphi|$ it follows
	\begin{equation}\label{AT-lift}
\ATeps(u,v)=\int_\Omega\left(v^2|\nabla\varphi|^2+ \var|\nabla v|^2+\frac{(v-1)^2}{4\var}
\right)\dx\,,
	\end{equation}
which is the Ambrosio-Tortorelli functional for the pair $(\varphi,v)$.
	Now,  if for any sequence $(u_{\var_k},v_{\var_k})\to(u,1)$ in $L^1(\Omega;\mathbb{S}^1)\times L^1(\Omega)$ we would be able to prove that $\varphi_{\var_k}\to\varphi$ in $L^1(\Omega)$, with $\varphi_{\var_k}\in W^{1,2}(\Omega)$ lifting of $u_{\var_k}$, we could conclude by applying  the classical results
\cite{AT90,AT92}, since the limit
$\varphi$ is a  lifting of $u$.  Unfortunately this is in general not true, as the energy provides a control  of $\nabla\varphi_{\var_k}$ only in regions where $v_{\var_k}$ is far from zero. Moreover a sequence of liftings might escape to infinity (since we can always add integer multiples of $2\pi$).
For this reason,  the main issue concerning the lower bound  in Theorem \ref{thm2} is to get a compactness result on sequences of liftings.
 This leads us to another main result contained in Theorems \ref{thm:compactness&lsc} and  \ref{teo:compactness}, which provides compactness and lower semi-continuity for a sequence of liftings associated to a single map $u$ and more generally, to a converging sequence $(u_{\var_k})$.  
 We stress that these compactness results have a local feature, in the sense that liftings converge up to locally subtract a suitable constant. Similar compactness results obtained by slight modifying the given sequence of functions can be found in \cite{manuel} (see also \cite{CC} where the compactness result in $GSBD^p$ is used in the proof of our result). 
 Theorems \ref{thm:compactness&lsc}  and  \ref{teo:compactness}, besides being interesting by themselves, allow us, together with a refined local argument, to show the lower bound inequality in Theorem \ref{thm2}. Furthermore  Theorem \ref{thm:compactness&lsc} 
  provides existence of a solution to \eqref{eq:minprob_intro}, that is, the infimum in \eqref{eq:minprob_intro} is actually a minimum.
  More precisely, in Corollary \ref{cor:existence-minimizer}, we show that the following more general minimum problem has a solution. Let $p>1$ and $u\in SBV^p(\Omega;\mathbb{S}^1)$; then
  there exists a lifting $\varphi_{\text{{\rm min}}}\in GSBV^p(\Omega)$ of $u$ such that 
  \begin{align}\label{min-prob_intro}
  	\mathcal H^{n-1}(S_{\varphi_{\text{{\rm min}}}})=
  	\inf
  	\left\{\mathcal H^{n-1}(S_\varphi):\varphi\in GSBV^p(\Om),\;e^{i\varphi}=u\text{ a.e. in }\Om
  	\right\}\,.
  \end{align}
}
It is worth to notice that,
somehow surprisingly, the minimum in \eqref{eq:minprob_intro}
is not attained in the class $SBV(\Om)$, as a consequence
of an example discussed in Section \ref{sec:an_example}.   This 
shows that the analysis of \eqref{eq:minprob_intro}
is rather delicate.

The problem of finding a lifting minimizing the length of the jump set somehow resembles the related question of finding a lifting minimizing its $BV$-seminorm \cite{Brezis-Mironescu}. As for the latter, which is strongly linked with Plateau and optimal transport problems \cite{Brezis-Mironescu2}, the structure of a jump minimizing lifting might be related to  optimal transport questions with different cost functions. We discuss in more details this issue in Subsection \ref{sec:optimal}.
\\

\noindent
{{\bf Further directions and open problems.} 
	Lifting theory is useful in several, apparently unrelated, contexts where  \textit{topological singularities} arise, such as screw dislocations in crystals, vortices in superconductors, the non-parametric Plateau problem in codimension-two and optimal transport problems. 
	This shows similarities between different problems to which a suitable version of the Ambrosio-Tortorelli approximations for $\mathbb{S}^1$-valued maps might be addressed.
 In the following we briefly review some open problems which will be investigated in future works.\smallskip
	
	\noindent
\textit{Dislocations and vortices}. Dislocations are line-defects in metals that locally alter the crystalline structure and are the main source of plastic slips. From a mathematical point of view they can be identified with codimension-two singularities. 
	In particular, in a simplified framework one can consider two-dimensional semi-discrete models for dislocations either of screw or edge type, which correspond to point singularities in a two-dimensional domain (see \cite{ADGP-Metastability,DGP,GLP,Ponsiglione,DLSVG, CDLS} and references therein). It is well known that  models for screw dislocations share similarities with Ginzburg Landau models for superconductors \cite{ACP}. In particular, in recent works \cite{DLSVG, CDLS} it was proposed a free-discontinuity model for screw dislocations.  The latter is given by an energy of Mumford-Shah type for $\mathbb{S}^1$-valued maps which penalizes the measure of the jump set. From the analysis pursued in \cite{DLSVG, CDLS} it turns out that such a model is equivalent (to the leading order) to the Ginzburg Landau one. This seems to suggest that an Ambrosio Tortorelli approximation for screw dislocations, as well as for vortices in superconductors, is possible. In addition, this offers the possibility to extend the results of \cite{DLSVG, CDLS}  to the three-dimensional setting (see \cite{CGO15,CGMarz,GMS} for models in dimension three). 
\smallskip

\noindent 
\textit{The Plateau and optimal transport problems}.  The Cartesian Plateau problem consists in finding an area-minimizing surface among all Cartesian surfaces spanning an appropriate Jordan space curve. While the codimension-one setting has been exhaustively understood \cite{Giusti} very few is known in codimension two. This has to do with the fact that the relaxed Cartesian area in codimension larger than one is not subadditive, and thus has no integral representation. 
A couple of  examples in codimension two where the computation of the relaxed graph area is possible are the vortex map $u_V$ and the triple junction function \cite{Acerbi-DalMaso}. Both are $\mathbb S^1$-valued maps with singularities, and the corresponding  relaxed area has a nonlocal feature involving  liftings,  similarly as in  \eqref{MSlif} (cf. \cite{BP,Scala,BEPS,BES} and references therein). Also, the singular contribution appearing in the case of $\mathbb S^1$-valued maps is related with optimal transport problems with different cost functions. In the specific case of $\mathbb S^1$-valued Sobolev maps, the singularities should be optimally connected each other, in a similar fashion as for the minimal connection of dipoles in the minimization of the $BV$-seminorm of liftings \cite{Brezis-Mironescu2} (see also \cite{SS,BSS}).  Connecting the singularities in the relaxed Cartesian area  is related with  finding an optimal Cartesian vertical current filling the holes of the graph of the given map $u$, and this often requires  the study of a Plateau problem in codimension-one with partial free boundary \cite{BMS}. 

In the  simplified setting of $\mathbb S^1$-constrained relaxation \cite[pg.~611~eq.~4]{Giaq}, the optimal value of the graph area has a singular contribution  characterized in terms of $\mathbb S^1$ vertical Cartesian currents \cite[pg.~612, Theorem 1]{Giaq} and that is again related to an optimal transport question on how to connect the singularities of $u$. A future development of the present research is  to approximate the relaxed area functional on $\mathbb S^1$-valued functions by an Ambrosio-Tortorelli energy where the bulk term has linear growth in the gradient (in the same spirit of \cite{ABS}).

We emphasize that phase-field models based on the Ambrosio-Tortorelli functionals have been already employed to study problems related with optimal transport questions (see e.g. \cite{FDW} and references therin, and \cite{CFM} for the relation with the steiner tree problem). While the link of Corollary \ref{cor:existence-minimizer} with the optimal transport question involves the cost function $\psi\equiv 1$ (see Subsection \ref{sec:optimal}), we believe that the analysis of the Ambrosio-Tortorelli functional for $\mathbb S^1$-valued maps with linear growth (instead of quadratic) would give rise to a minimization question similar to \eqref{min-prob_intro} involving different cost functions. Also this will be object of future development.
\smallskip
}

\noindent
%
{\bf Content of the paper.} {In Section \ref{sec:notation_and_main_definitions} we collect some notation and recall some useful tools which will be employed to prove the main results. In Section \ref{sec:existence_of_a_lifting_of_minimal_jump} we provide an example where the jump minimizing lifting is not in $SBV(\Omega)$ but just in $GSBV(\Omega)$. Furthermore, we describe a connection with optimal transport,
 we state the two compactness results
Theorem 
\ref{thm:compactness&lsc}, Theorem \ref{teo:compactness} and the existence of a minimizer to \eqref{eq:minprob_intro}, Corollary \ref{cor:existence-minimizer}.
In Section \ref{sec:proofs-compactness} we provide the proofs of Theorems \ref{thm:compactness&lsc} and \ref{teo:compactness}.
}
Eventually in Section \ref{sec:Gamma_convergence}, after establishing  some density and approximation results
in $SBV(\Omega)$, both for $\Suno$-valued functions
and for liftings, 
 we prove Theorem  \ref{thm1}; in the proof we do not
make use of 
the results of Section \ref{subsec:some_density}, but we 
utilize the results of Section \ref{sec:existence_of_a_lifting_of_minimal_jump}.
Finally, Section \ref{subsec:proof_2} is devoted to the proof of Theorem \ref{thm2}.

	\section{Notation and preliminaries}\label{sec:notation_and_main_definitions}
		In this section we collect some notation, 
and recall some notions concerning $SBV$ and $GSBV$ functions \cite{AmFuPa}
and lifting theory \cite{Brezis-Mironescu}. In what follows:
	\begin{enumerate}[label=(\alph*)]
		\item[-] $n\ge1$ is a fixed integer and $p>1$ is a fixed real number; 
		\item[-] 
$\partial^*A$ denotes the reduced boundary of finite perimeter set $A\subset\R^n$ ;
		\item[-] $\vert \cdot \vert$ and 
 $\mathcal H^{n-1}$ denote the Lebesgue measure and the $(n-1)$-dimensional Hausdorff measure 
in $\R^n$, respectively;
		\item[-] $\chi_A$ denotes 
the characteristic function of the set $A\subset \Rn$;
		\item[-] for $a,b\in \R^n $ the symbol $a
		\otimes b$ denotes the tensor product between $a$ and $b$;
		\item[-] $\mathbb{S}^1:=\{(x,y)\in\R^2\colon x^2+y^2=1\}$
is the unit circle in $\R^2$.
 	\end{enumerate}
	\medskip

	\subsection{$SBV$ and $GSBV$ functions}\label{subsec:SBV_and_GSBV_functions}
Let $\openset \subset \Rn$ be open and bounded, and $m
\geq 1$ an integer. 
We denote by $BV(\openset;\R^m)$ the space of vector-valued
\textit{functions with bounded variation in} $\openset$, and 
with $|\cdot|_{BV}$ 
and $\|\cdot\|_{BV}$  
the $BV$ seminorm and norm, respectively, i.e.
	$$|u|_{BV}:=|Du|(\openset),\qquad \|u\|_{BV}:=\|u\|_{L^1}+|u|_{BV},$$
	see \cite{AmFuPa}.
We say that $u\in BV(\openset;\R^m)$ 
belongs to the space of \textit{special functions with bounded variation in $\openset$} , 
i.e., $u\in SBV(U;\R^m)$, if its distributional gradient is a finite $\mathbb{R}^{m\times n}$- valued Radon measure without 
Cantor part, that is, 
	\begin{equation*}
		Du=\nabla u\mathcal L^n+\jump{u}\otimes \nu_u\mathcal{H}^{n-1}\res S_u\,,
	\end{equation*}
	where $\nabla u$ is the approximate gradient of $u$, $S_u$ is the approximate jump set of $u$, $\jump{u}=u^+-u^-$ is the jump opening and $\nu_u$ is the unit normal, see \cite[Def. 3.67]{AmFuPa}  
to $S_u$. A measurable function $u: U \to \R^m$ belongs to the space of \textit{generalised special functions with bounded variation in $U$}, that is, $u\in GSBV(U;\R^m)$, if $\phi\circ u\in {SBV_{\rm loc}(U)}$
for any 
$\phi\in C^1(\R^m)$ with $\nabla \phi$ compactly supported.\footnote{Recall that $f\in SBV_{\rm loc}(U)$ if $f\in SBV(K)$ for every $K\subset U$ compact.}
If $m=1$ we write $BV(U)=BV(U;\R)$, $SBV(U)=SBV(U;\R)$ and $GSBV(U)=GSBV(U;\R)$.
	\begin{remark}[\textbf{Equivalent definition of $GSBV$ for $m=1$}]
	\rm $GSBV(U)$ can be equivalently defined as the space of measurable functions $u\colon U\to\R$
 such that  $u\wedge M\vee (-M)\in{ SBV_{\rm loc}}(U)$ for any $M>0$.
	\end{remark}
 For $p>1$ we set
	$$
	SBV^p(U;\R^m) = \{u \in SBV(U;\R^m) : \nabla u \in L^p(U;\R^{m\times n})\text{ and }\mathcal H^{n-1}(S_u) < +\infty\},
	$$
	and 
		$$
	GSBV^p(U;\R^m) = \{u \in GSBV(U;\R^m) : \nabla u \in L^p(U;\R^{m\times n})\text{ and }\mathcal H^{n-1}(S_u) < +\infty\}.
	$$
	Again, when $m=1$ we write $SBV^p(U)=SBV^p(U;\R^m)$ and $	GSBV^p(U) =	GSBV^p(U;\R^m) $.
	{
We set
			\begin{equation*}
			BV(U;\mathbb{S}^1)=\{u\in BV(U;\R^2)\colon |u|=1\,\text{ a.e. in }\,U\}\,,
		\end{equation*}
		\begin{equation*}
			SBV(U;\mathbb{S}^1)=\{u\in SBV(U;\R^2)\colon |u|=1\,\text{ a.e. in }\,\Omega\}\,,
		\end{equation*}
		and for $p>1$
		\begin{equation*}
			SBV^p(U;\mathbb{S}^1)=\{u\in SBV(U;\mathbb{S}^1)\colon \nabla u\in L^p(U;\R^{2\times 2}),\, \mathcal{H}^{n-1}(S_u)<+\infty\}\,.
		\end{equation*}
	}
Eventually,
a (finite or countable) family $(E_i)$ of finite perimeter subsets of a finite perimeter set $F$ is called a  Caccioppoli partition of $F$ 
if the sets $E_i$ are pairwise disjoint, and their union is $F$.
The next technical observation will be needed later (Section \ref{subsec:proof_1}).
\begin{remark}[\textbf{Approximation of a $BV$ function by smooth functions}]\label{rem:smooth_approx}
	{\rm Let $A\subset\R^n$ be a
bounded open set with Lipschitz boundary and
$\varphi\in BV(A)$.
By \cite[Theorem 3.9]{AmFuPa}, 
for any $\delta>0$,  there exists a function $\varphi_\delta\in C^\infty(A)$ such that
\begin{equation*}
\int_A|\varphi-\varphi_\delta|\dx<\delta\,, \quad
\int_A|\nabla \varphi_\delta|\dx\le |D\varphi|(A)+\delta\,.
\end{equation*}
Moreover, by inspecting the proof, 
$$
\varphi= 
{\varphi_\delta} \quad \mathcal H^\onecod-\text{a.e. on }\partial A\,,$$
in the sense of $BV$-traces on $\partial A$ (see e.g. \cite[page 181]{AmFuPa}).  To see this last property we recall that $$\varphi_\delta:=\sum_{h\ge 1}(\varphi\psi_h)*\rho_h\,,$$  where
\begin{itemize}
 \item $(\psi_h)_{h\ge1}$ is a partition of unity relative to the covering $(A_h)_{h\ge1}$ of $A$ defined as
 \begin{equation*}
A_1:=\{x\in A\colon \dist(x,\partial A)>2^{-1}\}\,,
 \end{equation*}
 \begin{equation*}
	A_h:=\{x\in A\colon (h+1)^{-1}<\dist(x,\partial A)<(h-1)^{-1}\}\quad\text{ for } h\ge 2\,;
\end{equation*}
\item $(\rho_h)_{h\ge1}$ is a family of mollifiers 
such that 
\begin{equation*}
\supp((\varphi\psi_h)*\rho_h)\subset A_h\,,
\end{equation*}
\begin{equation*}
	\int_A\Big[|(\varphi\psi_h)*\rho_h-\varphi\psi_h|
+|(\varphi \nabla \psi_h)*\rho_h- \varphi
\nabla\psi_h|
	\Big]\dx<2^{-h}\delta\,.
\end{equation*}
\end{itemize}
Thus, for $\mathcal H^\onecod$-a.e. $x_0\in\partial A$, we have
\begin{equation}\label{ex:nonse}
	\begin{split}
|\varphi(x_0)-\varphi_\delta(x_0)|
&\le \lim_{r\searrow0} \frac2{\omega_n} \frac1{r^\dimension}\int_{A\cap B_r(x_0)}\left| \varphi(y)-\varphi_\delta(y)
\right|\dy\\
&\le  \lim_{r\searrow0} \frac2{\omega_n}\frac1{r^\dimension}\int_{A\cap B_r(x_0)}
 \sum_{h\ge 1}\Big|
\varphi(y)\psi_h(y)-( \varphi\psi_h)*\rho_h(y)
\Big|\dy\,.
	\end{split}
\end{equation}
Now we observe that $A_h\cap B_r(x_0)\ne\emptyset$ if and only if $\frac1{h+1}<r$, i.e., $h>\frac1r-1$. Hence for $y\in B_r(x_0)$,
the sum on the right hand side of \eqref{ex:nonse} reduces to 
\begin{equation*}
\sum_{h\ge \frac1r-1}\Big|
\varphi(y)\psi_h(y)-( \varphi\psi_h)*\rho_h(y)
\Big|\le  \sum_{h\ge \frac1r-1}2^{-h}\delta = 2^{-\frac1r+2}\delta\,.
\end{equation*}
 We conclude 
\begin{equation*}
|\varphi(x_0)-\varphi_\delta(x_0)|\le \lim_{r\searrow0} \frac2{\omega_n}\frac1{r^\dimension}
 2^{-\frac1r+2}\delta r^\dimension=0\,.
\end{equation*}}
\end{remark}
\subsection{$\Gamma$-approximation and compactness for the Mumford-Shah functional}
We recall two classical results
\cite{AT90,AT92} (see also \cite{Focardi} for the generalization to the 
vector  case), and \cite[Theorem 4.8]{AmFuPa}. 
\begin{theorem}[\textbf{Ambrosio-Tortorelli}]\label{thm:Ambrosio-Tortorelli} 
Let $\opensetAT
\subset\R^n$ be an open set 
with Lipschitz boundary, and 
$$
\domAT
:=\{(u,v)\in W^{1,2}(\opensetAT;\R^m)
\times W^{1,2}(\opensetAT)\colon 0\le v\le1\}.
$$
	Then the functionals 
	\begin{equation*}
	\ATeps(u,v):=\begin{cases}	\displaystyle
		\int_{\opensetAT}\left(v^2|\nabla u|^2+\varepsilon|\nabla v|^2+\frac{(v-1)^2}{4\varepsilon}\right){\rm d}x&\text{ if }(u,v)
\in\domAT\,,\\[1em]
		+\infty& \text{otherwise in }L^1(\opensetAT;\R^m)
\times L^1(\opensetAT)\,,
	\end{cases}
	\end{equation*}
$\Gamma-L^1$-converge
to the Mumford-Shah functional
	\begin{equation*}
	\MS(u,v):=\begin{cases}
		\displaystyle \int_\opensetAT|\nabla u|^2\dx+\Hausdorff(S_u)
&\text{ if }u\in GSBV^2(\opensetAT;\R^m),\, v=1\text{ a.e.}\\[1em]
		+\infty & \text{otherwise in }L^1(\opensetAT;\R^m)
\times L^1(\opensetAT)\,,
	\end{cases}
	\end{equation*}
as $\var\to0^+$.
\end{theorem}
\begin{remark}\label{rem:decoupling}\rm
	By inspecting the proof of Theorem \ref{thm:Ambrosio-Tortorelli} 
one actually deduces the following  properties:
\begin{enumerate}
\item[(i)] \textbf{Decoupling.} Let $\var_k\searrow0$, and let $((u_k,v_k))_{k\ge1}\subset \domAT$
be a sequence converging to $(u,1)$ 
in $L^1(\opensetAT;\R^m)\times L^1(\opensetAT)$ with 
$\sup_{k \in \NN}\ATepsk(u_k,v_k)<+\infty$. Then 
\begin{equation}\label{eq:decoupling}
	\begin{split}
		&	\liminfk \int_\opensetAT
		v_k^2|\nabla u_k|^2\dx\ge \int_\opensetAT|\nabla u|^2\dx\,,\\[1em]
		&\liminfk\int_\opensetAT \left(\varepsilon_k|\nabla v_k|^2+
		\frac{(v_k-1)^2}{4\varepsilon_k}\right)\dx\ge \Hausdorff(S_u)\,.
	\end{split}
\end{equation}
\item[(ii)] {\textbf{$\Gamma$-convergence on a larger domain.} The result  still holds if one replaces $\mathcal{A}$ with the larger class 
$$
\widehat\domAT
:=\{(u,v)\in W^{1,1}(\opensetAT;\R^m)
\times W^{1,2}(\opensetAT)\colon v|\nabla u|\in L^2(\Omega),\,0\le v\le1\}.
$$
}
\end{enumerate}

\end{remark}

\begin{theorem}[\textbf{Compactness in $SBV$}]\label{teo:ambrosio}
	Let $p>1$, $\Omega\subset\ambientsource$ be a bounded open set and $(\varphi_k)_{k\ge1}
\subset SBV(\Omega)$ be a sequence
satisfying
$$
\sup_{k \in \NN}
\Big(\int_\Omega|\nabla \varphi_k|^p\dx+\Hausdorff(S_{\varphi_k})
\Big)<+\infty,
$$
and
\begin{equation}\label{eq:uniform_bound_in_Linfty}
\sup_{k \in \NN}
\Vert \varphi_k\Vert_\infty < +\infty.
\end{equation}
Then there exists a subsequence of $(\varphi_k)$ 
 weakly-star
converging in $BV(\Omega)$ 
to a 
function belonging to $SBV(\Omega)$.
\end{theorem}

Removing assumption \eqref{eq:uniform_bound_in_Linfty}
leads to the following result, which 
is the $GSBV^p$ variant  of \cite[Theorem 1.1]{CC},
and that will be applied in the proof of Lemma \ref{lem:passo_base} to an 
appropriate 
sequence of liftings.

\begin{theorem}[\textbf{Compactness}]\label{teo:vito}
	Let $p>1$, $\opensetcompactness\subset\ambientsource$ 
be an open set and $(\funzioni_vito_k)_{k\ge1}
\subset { G}SBV^p(\opensetcompactness)$ be a sequence
satisfying
$$
\sup_{k \in \NN}
\Big(
\int_\opensetcompactness
|\nabla \funzioni_vito_k|^p\dx+\Hausdorff(S_{\funzioni_vito_k})
\Big)
<+\infty. 
$$
Then there exist a (not-relabelled) subsequence and  a function $\varphi_\infty\in GSBV^p(\opensetcompactness)$ such that:
	\begin{equation*}
		\begin{split}
&
		E:=\{x\in \opensetcompactness:
|\funzioni_vito_k(x)|\rightarrow +\infty\} \text{
	has finite perimeter},
\\
& \varphi_\infty=0 \text{ on } E, 
\\
&\funzioni_vito_k\rightarrow \varphi_\infty\text{ a.e. on }
\opensetcompactness
\setminus E\,,\\
			&{ \nabla\funzioni_vito_k\to \grad \varphi_\infty \text{ in }
L^1(\opensetcompactness\setminus E;\R^\dimension)}\,,\\
			&\Hausdorff(S_{\varphi_\infty}\cup \partial^*E) \leq
\liminf_{k\rightarrow +\infty}\Hausdorff(S_{\funzioni_vito_k}).
		\end{split}
	\end{equation*}  
\end{theorem}

\subsection{Liftings of $\mathbb{S}^1$-valued maps}
\label{subsec:liftings_of_S1_maps}
 Let $\Omega\subset\R^n$ be a bounded connected and simply connected
open  set with Lipschitz boundary, and  $u\colon\Omega\to \mathbb{S}^1$ 
be a measurable function. A \textit{lifting} of $u$ is a measurable function $\varphi\colon\Omega\to\R$ such that 
 \begin{equation*}
 	u(x)=e^{i\varphi(x)}\quad \text{for a.e. }x\in\Omega\,.
 \end{equation*}
{
Given a Borel set $B\subseteq\Omega$, we say that $\varphi:B\rightarrow \R$ is a lifting of $u$ on $B$ if the previous equality holds a.e. on $B$. 
}Clearly, 
if $\varphi$ is a lifting of $u$, then so is $\varphi+2\pi m$ 
for all $m\in \mathbb{Z}$.

If $u$ has some regularity, a natural question is  whether $\varphi$ can be chosen with the same regularity. The answer is 
partially positive, 
see \cite{Brezis-Mironescu,Dav-Ign} for more details:
\noindent
\begin{enumerate}[label=$(\arabic*)$]
\item If $u\in C^k(\overline\Omega;\mathbb{S}^1)$ for some $k\ge0$, then
$u$ has 
a lifting $\varphi\in C^k(\overline{\Omega})$, unique (mod. $2\pi$),
\cite[Lemma 1.1]{Brezis-Mironescu};
\item If $u\in C^k(\overline\Omega;\mathbb{S}^1)\cap W^{1,p}(\Omega;\mathbb{S}^1)$ for some $p\in[1,+\infty]$, then $u$ has 
 a lifting $\varphi\in C^k(\overline{\Omega})\cap W^{1,p}(\Omega)$;
If $n=1$ and $u\in W^{1,p}(\Omega;\mathbb{S}^1)$ for some $p\in[1,+\infty)$, then $u$ has 
 a lifting $\varphi\in W^{1,p}(\Omega)$; 
\item If $n\ge2$ and $u\in W^{1,p}(\Omega;\mathbb{S}^1)$ for some $p\in[2,+\infty)$, then $u$ has 
 a lifting $\varphi\in W^{1,p}(\Omega)$. Moreover $\varphi$ is unique (mod $2\pi$);
\item If $n\ge 2$, then for every $p\in[1,2)$ there exists $u\in W^{1,p}(\Omega;\mathbb{S}^1)$ for which 
there are no liftings $\varphi\in W^{1,p}(\Omega)$,
see \cite[Theorem 1.2, Remark 1.9]{Brezis-Mironescu}.

\end{enumerate}
{A well-known example of property (4) when $n=2$ is the 
vortex map $u_V$ discussed in 
the introduction. 
Indeed it can be shown
{\cite[Pages 17-19]{Brezis-Mironescu}}
that there are no liftings of $u$ in $W^{1,1}(B_1)$
(and thus there are no liftings in $W^{1,p}(B_1)$ for all $p\in[1,2)$).\\
}

\noindent
Next we recall some regularity results on lifting of $BV$ maps.
\begin{theorem}[\textbf{Davila-Ignat}]\label{thm:dav-ign}
	Let $u\in BV(\Omega;\mathbb{S}^1)$. 
Then there exists a lifting $\varphi\in BV(\Omega)$ such that $\|\varphi\|_{L^\infty}\le2\pi$ and $|\varphi|_{BV}\le2|u|_{BV}$.
\end{theorem}
\begin{proof}
See \cite[Theorem 1.1]{Dav-Ign}, and 
also \cite[Theorem 1.4]{Brezis-Mironescu}.
\end{proof}

\begin{remark}\label{rem:dav-ign}\rm
If $u\in SBV(\Omega;\mathbb{S}^1)$ then $\varphi$ of Theorem \ref{thm:dav-ign} can be chosen in $SBV(\Omega)$. If $u\in SBV^p(\Omega;\mathbb{S}^1)$ for some $p>1$ then $\varphi$ of Theorem \ref{thm:dav-ign} can be chosen in $SBV^p(\Omega)\cap  L^\infty(\Omega; [0,2\pi])$.
\end{remark}

As usual, for any lifting $\varphi\in SBV(\Om)$ of $u$ 
we write\footnote{$S_\varphi^I$ stands for the ``integer'' part of the
jump,
and $S_\varphi^f$ for the ``fractional'' part.} 
$S_\varphi=S_\varphi^I\cup S_\varphi^f$ where 
$$
S^I_\varphi:=\{x\in S_\varphi:\jump{\varphi}(x)\in 2\pi\mathbb Z\},
\qquad  S_\varphi^f:=S_\varphi\setminus S_\varphi^I.
$$
Notice that in particular $S^f_\varphi=S_u$.  

Let $u\in BV(\Omega;\mathbb{S}^1)$ and consider the minimum problem
\begin{equation*}
\inf\{|\varphi|_{BV}\colon \varphi\in BV(\Omega),\, e^{i\varphi}=u\text{ a.e. in }\Omega\}.
\end{equation*}
Then there exists 
a minimizer $\varphi\in BV(\Omega)$ 
such that $|\varphi|_{BV}\le 2| u|_{BV}$ and $0\le \int_\Omega\varphi\dx\le 2\pi|\Omega|$ \cite[page 25]{Brezis-Mironescu}. 
As explained in the
introduction, we are instead concerned with the existence of a 
lifting which minimizes the measure of the jump set. This will be
the argument of the next section.

 \section{On jump minimizing liftings}
\label{sec:existence_of_a_lifting_of_minimal_jump}

In view of the applications in Section \ref{sec:Gamma_convergence},
we are concerned
with
the analysis of the following minimization problem, which has
an independent interest: let $p>1$, let $u\in SBV^p(\Omega;\mathbb{S}^1)$ and  define
\begin{align}\label{min-prob}
	\minvalueup
:=\inf
\left\{\mathcal H^{n-1}(S_\varphi):\varphi\in GSBV^p(\Om),\;e^{i\varphi}=u\text{ a.e. in }\Om
\right\}\,.
\end{align}
We observe that, in \eqref{min-prob},
 the set of $\varphi$
between parentheses is non empty, due to Remark \ref{rem:dav-ign}.
The reason of utilizing the space 
$GSBV^p(\Omega)$ instead that $SBV^p(\Omega)$ can be understood from the following example.

\subsection{An example}\label{sec:an_example}
In this section we will show that, in general, 
a lifting minimizing the right-hand side 
of \eqref{min-prob} 
cannot be found in $SBV^p(\Om)$, for any $p \geq 1$. The
strategy consists in constructing a real-valued
function $\varphi$ {such that, letting $u:=e^{i\varphi}$, the following hold:
\begin{itemize}
\item $\varphi\in GSBV^p(\Omega)\setminus SBV(\Omega);$
\item $S_\varphi
= S_u$;
\item $\Omega\setminus S_\varphi$ is arcwise connected.
\end{itemize}
In this way $\varphi$ is the only minimizer (modulo addition of a constant in $2\pi\mathbb{Z}$) of \eqref{min-prob} for $u$.
}

We consider $\Om:=(0,1)^2\subset\R^2$, and three sequences $(R_n)$, $(T_n)$,
$(B_n)$ 
of open rectangles contained in $\Om$ defined,
for an even integer $n\geq 2$, as follows:
$$
 \begin{aligned}
& R_n\text{ has vertices }\left(\frac{1}{2^{n+1}},0\right),\; \left(\frac{1}{2^{n-1}},0\right),\;\left(\frac{1}{2^{n-1}},\frac{1}{10^n}+\frac{1}{n^2}\right),\;\left(\frac{1}{2^{n+1}},\frac{1}{10^n}+\frac{1}{n^2}\right);
\\	
& T_n\text{ has vertices }\left(\frac{1}{2^{n+1}},\frac{1}{n^2}\right),\; \left(\frac{1}{2^{n-1}},\frac{1}{n^2}\right),\;\left(\frac{1}{2^{n-1}},\frac{1}{10^n}+\frac{1}{n^2}\right),\;\left(\frac{1}{2^{n+1}},\frac{1}{10^n}+\frac{1}{n^2}\right);
\\	
& B_{n+1}\text{ has vertices }\left(\frac{1}{2^{n+2}},0\right),\; \left(\frac{1}{2^{n}},0\right),\;\left(\frac{1}{2^{n}},\frac{1}{10^{n+1}}\right),\;\left(\frac{1}{2^{n+2}},\frac{1}{10^{n+1}}\right).
\end{aligned}
$$
Observe that the $R_n$'s are pairwise disjoint and 
the closures $\overline T_n$ of $T_n$ are pairwise disjoint.
Also the $B_n$'s are pairwise disjoint, 
but their boundaries share some part of the lateral edges 
(see Fig. \ref{fig1}). 
Furthermore,
for all $n, m\geq 2$ even,
$$
T_n\subset R_n, \quad
B_{n+1}\subset R_n\cup R_{n+2},
\quad
\overline T_n\cap \overline B_m=\varnothing.
$$

\begin{figure}
	\includegraphics[width=0.9\textwidth]{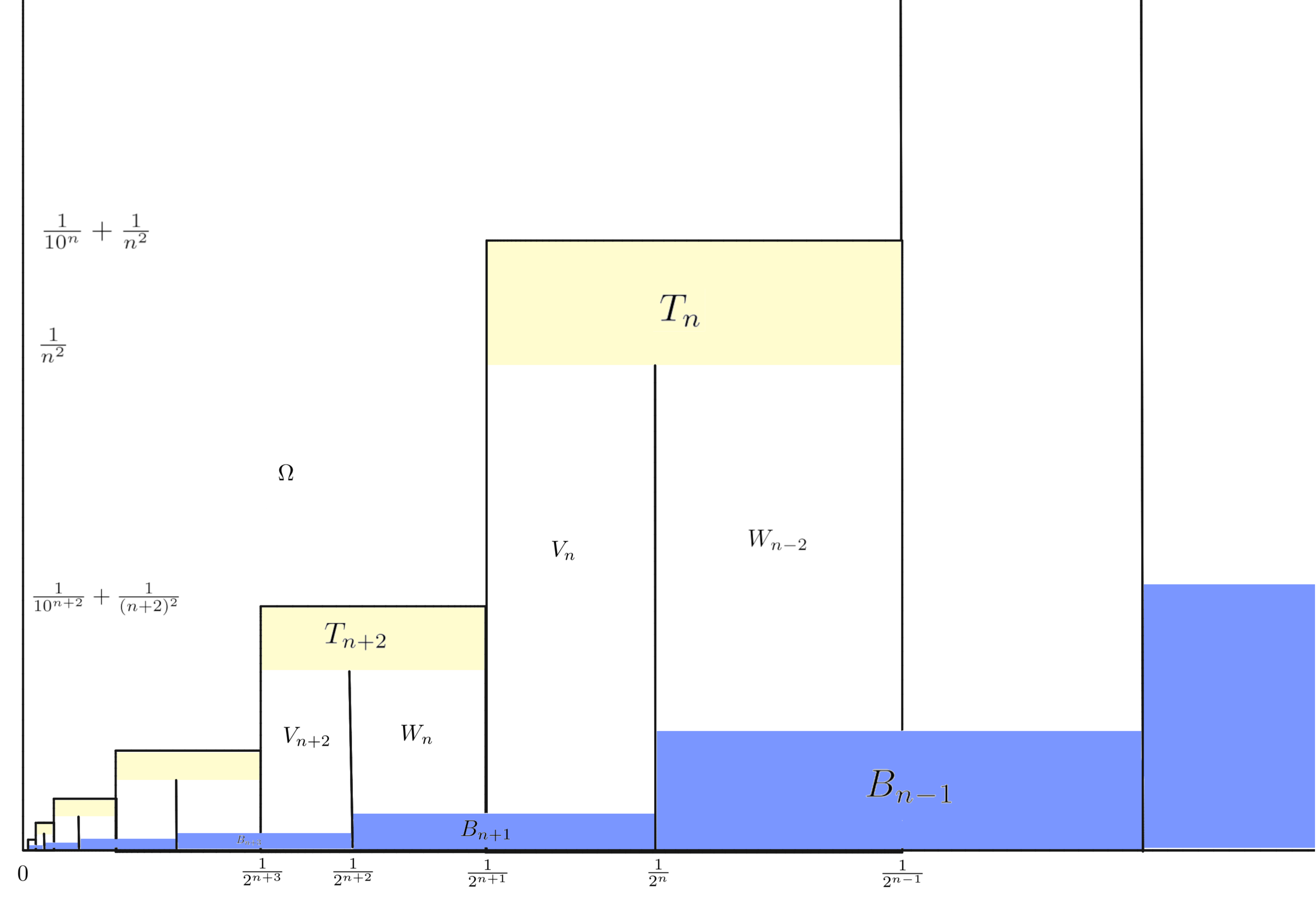}
	\caption{The rectangles 
$T_n$ and $B_n$ in 
a portion of $\Om=(0,1)^2$.
	}\label{fig1}
\end{figure}

Now, we define a map\footnote{The value 
	$4$ has not particular meaning: any positive constant could be chosen as well. A similar comment applies to \eqref{def-phi2}} $\liftexample:\Om\rightarrow \R$
\begin{align}\label{def-phi}
	\liftexample:=	\begin{cases}
		n^2&\text{in  }T_n\text{ for $n\geq 2$ even},\\
		(n+1)^2&\text{in  }B_{n+1}\text{ for $n\geq 2$ even},
\\		4&\text{in } \Omega \setminus 
\bigcup_{\substack{n\geq2\\n\text{ even}}}R_n.
\end{cases} 
\end{align}
It remains to define $\liftexample$ on 
$$
\displaystyle 
\bigcup_{\substack{n\geq2\\n\text{ even}}}R_n\setminus 
(\bigcup_{\substack{n\geq2\\n\text{ even}}} (T_n\cup B_{n+1})).
$$
 To this purpose we set
  \begin{align}\label{def-phi2}
 \liftexample:=4 \text{ in the open rectangle with vertices }\left(\frac{1}{4},0\right),\; \left(\frac{1}{2},0\right),\;\left(\frac{1}{2},\frac{1}{4}+\frac1{10^2}\right),\;\left(\frac{1}{4},\frac{1}{4}+\frac1{10^2}\right).
 \end{align}
Further, we consider two sequences 
$(V_n)$, $(W_n)$ of open
rectangles, defined for any even $n\geq 2$, as follows:
$$
  \begin{aligned}
&	V_n\text{ has vertices }\left(\frac{1}{2^{n+1}},\frac{1}{10^{n+1}}\right),\; \left(\frac{1}{2^{n}},\frac{1}{10^{n+1}}\right),\;\left(\frac{1}{2^{n}},\frac{1}{n^2}\right),\;\left(\frac{1}{2^{n+1}},\frac{1}{n^2}\right),
\\
&	W_n\text{ has vertices }\left(\frac{1}{2^{n+2}},\frac{1}{10^{n+1}}\right),\; \left(\frac{1}{2^{n+1}},\frac{1}{10^{n+1}}\right),\;\left(\frac{1}{2^{n+1}},\frac{1}{(n+2)^2}\right),\;\left(\frac{1}{2^{n+2}},\frac{1}{(n+2)^2}\right),
\end{aligned}
$$
see Fig. \ref{fig1}.

By \eqref{def-phi} the value of $\liftexample$ on the top edge of $V_n$ is $n^2$, whereas on the bottom edge is $(n+1)^2$, and we define $\liftexample$ on $V_n$ as the 
affine interpolation of 
these two values.
The value of $\liftexample$ on the top edge of $W_n$ is $(n+2)^2$, whereas on the bottom 
edge is $(n+1)^2$, 
and  we define $\liftexample$ on $W_n$ as the 
affine interpolation of these two values.
Observe that, looking from top to bottom, $\liftexample$ increases in $V_n$ 
and decreases in $W_n$. Clearly, 
\begin{equation}\label{eq:not_bounded}
\liftexample 
\notin L^\infty(\Om),
\end{equation}

\begin{equation}\label{eq:piece}
\liftexample 
\textrm{ is piecewise constant in } \Om \setminus (\bigcup_{\substack{n\geq2\\n\text{ even}}} 
(V_n \cup W_n)), 
\textrm{ piecewise affine in } \bigcup_{\substack{n\geq2\\n\text{ even}}} V_n \cup W_n, 
\end{equation} 
and
\begin{equation*}
	\begin{split}
&|\nabla \liftexample|=
\Big|\frac{\partial \liftexample}{\partial x_2}\Big|
= \frac{(n+1)^2-n^2}{\frac1{n^2}-\frac1{10^{n+1}}}\approx n^3\quad \text{ on }V_n,\\
&
|\nabla \liftexample|=
\Big
|\frac{\partial \liftexample}{\partial x_2}\Big| =\frac{(n+2)^2-(n+1)^2}{\frac{1}{(n+2)^2}-\frac1{10^{n+1}}}
\approx  (n+1)^3 \quad \text{ on } W_n,
	\end{split}
\end{equation*} 
being
the height of $V_n$ (resp. $W_n$)
 of order $\frac{1}{n^2}$ (resp. $\frac{1}{(n+1)^2}$). However, the 
bases 
 of $V_n$ and $W_n$ have length $\frac{1}{2^{n+1}}$ and $\frac{1}{2^{n+2}}$ respectively, so 
 \begin{equation*}
 	\begin{split}
\|\nabla \liftexample\|^p_{L^p(V_n)}\approx|v_n|n^{3p}= \frac{n^{3p-2}}{2^{n+1}} 
=: a_n,
\quad \|\nabla \liftexample\|^p_{L^p(W_n)}\approx |w_n|(n+1)^{3p}= \frac{(n+1)^{3p-2}}{2^{n+2}} =: b_n.
 	\end{split}
 \end{equation*}
In particular 
\begin{equation}\label{eq:gradiente_integrabile}
\nabla \liftexample\in L^p(\Om; \R^2) \qquad \forall p\geq 1,
\end{equation}
 since
$\sum_n a_n <+\infty$ and $\sum_n b_n < +\infty$.
On the other hand, denoting $\Sigma_T$ the union (for $n\geq 4$ even) 
of the top and left edges of the $T_n$'s, by $\Sigma_B$ 
the union (for $n\geq 4$ even) of the lateral edges of the $B_n$'s, 
and  by $\Sigma_{VW}$ the union of the lateral edges of all $V_n$'s 
and $W_n$'s, we see that the jump set $S_\liftexample$ of $\liftexample$ is exactly
$$
S_\liftexample = \Sigma_T\cup \Sigma_B\cup \Sigma_{VW};
$$ 
furthermore,
\begin{equation}\label{eq:arcwise_connected}
\Om \setminus S_\liftexample \text{ is arcwise connected},
\end{equation}
and
it is easy to check that 
\begin{equation}\label{eq:salto_finito}
\mathcal H^1(S_\liftexample) < +\infty.
\end{equation}
 In particular, from \eqref{eq:piece}, \eqref{eq:salto_finito} and 
\eqref{eq:gradiente_integrabile} we deduce $\liftexample\in GSBV^p(\Om)$, for any $p \geq 1$.

 However,
$$
\liftexample\notin SBV(\Om).
$$
 Indeed, although $D\liftexample$ has no Cantor part, 
the total variation $|D^j\varphi|(\Omega)$ of the jump part $D^j\liftexample$ 
of $D\liftexample$ is infinite. This can be seen by observing that, 
on each lateral edge of $V_n$, the jump opening $\jump{\liftexample}$ 
is of order $n$  on a segment of length of order $n^{-2}$. Thus 
 $$|D^j\liftexample|(\Om)\geq \sum_{\substack{n\geq2\\n\text{ even}}}n^{-1}=+\infty.$$  

We claim that $\liftexample$ is the unique (up to addition of a constant)
solution of \eqref{min-prob} {for}
$$
u:=e^{i\liftexample}.
$$
We have $u \in BV(\Om; \circle)$ 
and $\vert \grad u\vert = \vert \grad \liftexample\vert$.
Now, we observe that 
\begin{equation}\label{eq:stesso_salto}
S_u = S_\liftexample,
\end{equation} 
namely
$\mathcal H^1(S_u\Delta S_\liftexample)=0$,
and therefore
$$u\in SBV^p(\Om;\mathbb S^1).$$
Equality \eqref{eq:stesso_salto}
is due to the fact that the subset of $S_\liftexample$ 
where $\jump{\liftexample}\in 2\pi\mathbb Z$ is $\mathcal H^1$-negligible. 
This additionally implies that any lifting $\psi$ of $u$ must satisfy
$S_\psi\supseteq S_u$, and so 
\begin{equation}\label{eq:Spsi_Sphi}
\mathcal H^1(S_\psi)\geq \mathcal H^1(S_\liftexample).
\end{equation}
We conclude from \eqref{eq:stesso_salto} and \eqref{eq:Spsi_Sphi}
that $\liftexample$ is a minimizer of \eqref{min-prob}, and 
$\liftexample \notin 
SBV^p(\Om)$, for any $p\geq1$. Finally, from 
\eqref{eq:arcwise_connected} it follows that $\liftexample$
is unique, modulo addition of a constant in $2\pi\mathbb Z$.

\subsection{A connection with optimal transport: an example}\label{sec:optimal}

We observe in this subsection that the minimization problem \eqref{min-prob} has a strict relation with a question arising from optimal transport (see \cite{MMS,MMT} for the setting and similar formulations). To do so, we give an example in a  special case, leaving the more general ones to future treatments: for a  connected and simply-connected bounded open set $\Om\subset\R^2$, we fix $1<p<2$, $N\in\mathbb N$, and a Sobolev map $u\in W^{1,p}(\Om;\mathbb S^1)$ with distributional Jacobian determinant 
\begin{align}\label{detnablau}
	\text{{\rm Det}}(\nabla u)=\pi\sum_{i=1}^N(\delta_{x_i}-\delta_{y_i}),
\end{align}
(see \cite{Brezis-Mironescu}).
We set $\mu:=\sum_{i=1}^N\delta_{x_i}$ and $\nu:=\sum_{i=1}^N\delta_{y_i}$, where the points $x_i$'s and $y_i's$ belong to $\Omega$ are not necessarily distinct, and we consider the class of all integer multiplicity $1$-currents whose boundary is $\mu-\nu$, namely
$$\mathcal T(\mu,\nu):=\{T\in \mathcal D_1(\Om):T=(R,\theta,\tau)\text{ is an i.m.c. such that }\partial T=\mu-\nu\}.$$
Here $T=(R,\theta,\tau)$ is the integer multiplicity current given by 
$$T(\omega)=\int_R\theta(x)\langle \omega(x),\tau(x)\rangle d\mathcal H^1\qquad \qquad \forall \omega\in \mathcal D^1(\Om),$$
with $R\subset\Omega$ a $1$-rectifiable set, $\tau$ a tangent  $1$-vector to it, and $\theta:R\rightarrow\mathbb Z$ a $\mathcal H^1$-measurable function. 
A classical optimal transport problem can be formulated on the class $\mathcal T(\mu,\nu)$ in the following way: one fixes a cost function $\psi:\mathbb Z\times \mathbb S^1\rightarrow \R^+$ and study the minimization problem
\begin{align}\label{trans1}
	\Psi(\mu,\nu):=\inf\Big\{\int_R\psi(\theta,\tau)d\mathcal H^1:(R,\theta,\tau)\in \mathcal T(\mu,\nu)\Big\}.
\end{align}
We refer to 
In the special case $\psi\equiv1$ we have $\Psi(\mu,\nu)=\int_R\psi(\theta,\tau)d\mathcal H^1=\mathcal H^1(R)$ and
we claim that
\begin{align}\label{trans=jump}
	m_p[u]=\inf\{\mathcal H^1(R):(R,\theta,\tau)\in \mathcal T(\mu,\nu)\},
\end{align}
which shows the connection between our minimization problem \eqref{min-prob} and optimal transport.
We sketch the main steps to prove this, leaving details to future developments. Let $\varphi\in SBV^p(\Om)$ be a lifting\footnote{We here suppose, for simplicity, that the infimum in \eqref{min-prob} can be obtained restricting on the space $SBV^p(\Om)$.} of $u$; then 
\begin{align}
	D\varphi=\nabla \varphi\;\mathcal L^2+D^j\varphi,
\end{align}
and since $D^j\varphi=\jump{\varphi}n_{S_\varphi} \mathcal H^1\res S_\varphi$, its rotated\footnote{We take the conterclockwise $\pi/2$-rotation.} $(D^j\varphi)^\perp=\jump{\varphi}n^\perp_{S_\varphi} \mathcal H^1\res S_\varphi$ can be identified with $2\pi$ times the $1$-current $$T_\varphi=(S_\varphi,\frac{\jump{\varphi}}{2\pi},n^\perp_{S_\varphi})$$ which has integer multiplicity.
We claim that $$\partial T_\varphi=\mu-\nu\quad \text{ in }\mathcal D_0(\Om);$$ indeed, if we identify $T_\varphi\in\mathcal D_1(\Om)$ with a vector-valued measure and the elements of $\mathcal D_0(\Om)$ with scalar measures, $\partial T_\varphi\in \mathcal D_0(\Om)$ corresponds to minus the divergence of $T_\varphi$, and so 
\begin{align}
	\partial T_\varphi=-\frac{1}{2\pi}\text{{\rm Div}}((D^j\varphi)^\perp)=\frac{1}{2\pi}\text{{\rm Div}}(\nabla ^\perp\varphi)
\end{align} where we have used that 
$$\text{{\rm Div}}(D^\perp\varphi)=\text{{\rm Curl}}(D\varphi)=0.$$
Here ${\rm Curl}(D\varphi)$ is defined in a distributional sense.
Recalling the definition of distributional Jacobian determinant for Sobolev maps \cite[Page 12, formula (1.33)]{Brezis-Mironescu} one has 
\begin{align}
	\text{{\rm Det}}(\nabla u)=\frac12\text{{\rm Div}}(u_1\nabla^\perp u_2-u_2\nabla^\perp u_1)=\frac12\text{{\rm Div}}(\nabla^\perp\varphi)
\end{align}
so we conclude $\partial T_\varphi=\mu-\nu$.\\ In particular $T_\varphi$ is admissible for \eqref{trans1}, i.e., $T_\varphi\in \mathcal{T}(\mu,\nu)$, and we obtain
\begin{align}
	m_p[u]\geq\inf\{\mathcal H^1(R):(R,\theta,\tau)\in \mathcal T(\mu,\nu)\}.
\end{align}
To see that also the opposite inequality holds, let $T=(R,\theta,\tau)$ be admissible for \eqref{trans1}, and let us decompose $T=\sum_{i=1}^{+\infty }T_i$ in indecomposable components, so that by Federer decomposition theorem \cite[Sections 4.2.25 and 4.5.9]{Federer:69} $T_i$ is either a loop with multiplicity $\pm1$ or a Lipschitz parametrized path from one point $y_j$ to some $x_k$. Up to erasing  the loops (operation that does not increase the energy in \eqref{trans1}) we may suppose that there are exactly $N$ such Lipschitz paths, $T=\sum_{i=1}^NT_i$. In particular $T_i=(R_i,1,\tau)$ is such that $R_i$ is a closed set image of $[0,1]$ under the Lipschitz injective map $\gamma_i$, and so $R=\cup_{i=1}^N R_i$ is closed. Consider then the open set $\Om\setminus R$ and let $\varphi$ be a lifting of $u$ in   $\Om\setminus R$ with jump set of minimal length. We claim that $\mathcal H^1(S_\varphi\cap (\Om\setminus R))=0$; this is equivalent to say that there exists a lifting of $u$ on $\Omega\setminus R$ with no jumps. The latter can be shown by the following observation: we connect the components of $R$ with a curve $\sigma$ in such a way $\Om\setminus (R\cup\sigma)$ is simply-connected. Then by \cite[Lemma 1.8]{Brezis-Mironescu} the lifting $\varphi$ has no jumps\footnote{Although the domain $\Om\setminus( R\cup \sigma)$ is not Lipschitz, the same result can be obtained by approximating it with suitable Lipschitz subdomains.} on $\Omega\setminus (R\cup \sigma)$. To see that there is no jump of $\varphi$ on $\sigma$, it is enough to observe that,  given any closed loop $\Gamma$ in $\Omega\setminus R$, the topological degree of $u$ on $\Gamma$ must be null (by construction of $R$).

From the claim, since $\varphi\in SBV^p(\Om)$ satisfies $S_\varphi\subset R$, we easily infer
\begin{align}
	m_p[u]\leq\inf\{\mathcal H^1(R):(R,\theta,\tau)\in \mathcal T(\mu,\nu)\}
\end{align}
by the arbitrariness of $T$. 

In a similar manner, we see how the problem in \eqref{min-prob} and the aforementioned transport problem can be used to solve Steiner-type problems. Again, assume for simplicity that the singularities of  $u\in W^{1,p}(\Om;\mathbb S^1)$ satisfy \eqref{detnablau}. Assume also that $x_i=\overline x$ for all $i=1,\dots,N$, and that the points $y_i$ are all distinct and different from $\overline x$. We look for the connected set of minimal length containing the $N+1$ points in the family  $C=\{\overline x,y_i\colon i=1,\dots,N\}$. To do this, consider a compact connected set $K$ with $\mathcal H^1(K)<+\infty$, containing $C$ and let us suppose that ${\rm dist}(C,\partial \Om)>\mathcal H^1(K)$; then the Steiner problem for $C$ can be proven to be
\begin{align}
	m_p[u]=\inf\{\mathcal H^1(L):L\text{ is connected, }L\supset C\},
\end{align}  
and the jump set of a  jump minimizing lifting for $u$ is a minimizer of the right-hand side in the above expression.

\subsection{Main results on sequences of liftings}
The main results of the first part of 
the  paper are the following.
\begin{theorem}[\textbf{Compactness and lower semi-continuity}]\label{thm:compactness&lsc}
Let $p>1$, $u\in SBV^p(\Om;\mathbb S^1)$ and  
$(\varphi_k)_{k\ge1}\subset GSBV^p(\Om)$ be a  
sequence of liftings of $u$ with 
$$
\sup_{k \in \NN}
\Hausdorff(S_{\varphi_k})<+\infty. 
$$
Then 
there exist a Caccioppoli
partition $(E_i)_{i\ge1}$
 of $\Omega$ and
a not-relabelled subsequence of indices $k$ for which the following 
holds: 
\begin{itemize}
\item[(a)] 
there exists a lifting $\varphi_\infty\in GSBV^p(\Om)$ 
of $u$ in $\Omega$,
\item[(b)]
there exists a sequence
$(d^{(i)}_k)_{{k\ge1}}\subset \mathbb Z$ for any $i \in \NN$,
\end{itemize}
so that 
	\begin{equation}\label{eq:compactness}
\begin{split}
			& \lim_{k \to +\infty} 
(\varphi_k(x)-2\pi d^{(i)}_k) = \varphi_\infty(x)\qquad 
\forall i \in \NN, ~ 
\text{ for a.e. }x\in E_{i}\,,\\
			&\lim_{k \to +\infty} 
|\varphi_k(x)-2\pi d^{(i)}_k| = +\infty\qquad \forall i \in \NN, ~ 
\text{ for a.e. }x\in \Om\setminus E_{i}\,,
		\end{split}
	\end{equation}
and 
	\begin{equation}\label{eq:lsc}
		\liminf_{k\rightarrow +\infty}\Hausdorff(S_{\varphi_k})\ge 	\Hausdorff(S_{\varphi_\infty})\,.
	\end{equation}
\end{theorem}

\begin{cor}[\textbf{Existence}]\label{cor:existence-minimizer}
Let $p>1$.
Then there exists a minimizer $\varphi\in GSBV^p(\Omega)$ 
of \eqref{min-prob}.
\end{cor}
As pointed out in the introduction,
in general a minimizer of \eqref{min-prob} does not exist in $SBV^p(\Om)$.

The generalization of Theorem \ref{thm:compactness&lsc}
to a sequence $\suk$,
needed in the proof of Theorem \ref{thm2},  reads as follows.

	\begin{theorem}\label{teo:compactness}
Let $p>1$, $u,u_k\in SBV^p(\Om;\mathbb S^1)$ 
be such that 
\begin{equation}\label{eq:dododo}
u_k\stackrel{*}{\rightharpoonup}u \quad
{\rm in~} BV(\Om;\mathbb S^1),
\end{equation}
and let $\varphi_k\in GSBV^p(\Om)$ be a 
lifting of $u_k$, for all $k\ge 1$. Suppose   
		$$\sup_{k\in \mathbb{N}}\mathcal H^{n-1}(S_{\varphi_k})<+\infty\,.$$
		Then
there exist a Caccioppoli
partition $(E_i)$ of $\Omega$
and a not-relabelled subsequence of indices $k$ 
for which the following holds:
\begin{itemize}
\item[(a)] there exists
 a lifting $\varphi_\infty\in GSBV^p(\Om)$ of $u$ in $\Om$,
\item[(b)] there exist a
sequence $(\dki)_{k\ge1}\subset \mathbb Z$  for 
any $i\in\mathbb{N}$, 
\end{itemize}
so that 
		\begin{equation}\label{main_equations}
\begin{split}
&\lim_{k \to +\infty}(\varphi_k(x)-2\pi \dki)= \varphi_\infty(x)\qquad \forall
i \in \NN,~ \text{ for a.e. }x\in E_{i}\,,\\
				&\lim_{k \to +\infty}|\varphi_k(x)-2\pi 
\dki
|= +\infty\qquad 
\forall i \in \NN, ~
\text{ for a.e. }x\in \Om\setminus E_{i}\,,\\
			\end{split}
		\end{equation}
and
		\begin{equation}\label{eq:liminf}
			\liminf_{k\rightarrow +\infty}
\Hausdorff(S_{\varphi_k})\ge \Hausdorff(S_{\varphi_\infty})\,.
		\end{equation} 
	\end{theorem}

\section{Proofs of Theorems \ref{thm:compactness&lsc} and \ref{teo:compactness}} \label{sec:proofs-compactness}
Next Lemmas \ref{lem:passo_base} and 
\ref{lem:ricoprimenti}, independent one each other, are the building blocks of an iterative argument needed for 
the proof of Theorem \ref{thm:compactness&lsc}. Some arguments
in the proof of Lemma \ref{lem:ricoprimenti} will be also used in 
the proof of Theorem \ref{thm2}.

\begin{lemma}[\textbf{Localized compactness}]\label{lem:passo_base} 
Let $p>1$, let $\Omega\subset\R^\dimension$ be a bounded open set with Lipschitz 
boundary, 
and $u\in SBV^p(\Om;\mathbb S^1)$. 
 Let $(\varphi_k)_{k \geq 1}
\subset GSBV^p(\Om)$ be 
a sequence of liftings of $u$, and suppose
\begin{align}\label{CI}
\constlemmabase
:=\sup_{k \geq 1}\Hausdorff(S_{\varphi_k})<+\infty.
\end{align}
Let
	$F\subset\Om$ be a nonempty finite perimeter set
in $\Om$. 
	Then, for  a not-relabelled subsequence, 
there exist a sequence of integers $(d_k)_{k\ge1}\subset \mathbb{Z}$, a finite perimeter set 
$$
E \subseteq F
$$
in $\Om$, 
 and a function $\varphi_\infty\in GSBV^p(\Om)$, 
such that:
\begin{equation}\label{eq:phi_infty_is_a_lifting}
 \varphi_\infty 
\text{ is a lifting of } u \text{ in } E\,, \quad
\varphi_\infty=0\quad\text{ in }
\Om\setminus E\,,
\end{equation}

	\begin{equation}\label{eq:lem5.9}
		\begin{split}
		&	\lim_{k\to+\infty}(\varphi_k(x)-2\pi d_k)= \varphi_\infty(x)\qquad \text{ for a.e. }x\in E\,,
\\
			&  \lim_{k\to+\infty}	|\varphi_k(x)-2\pi d_k|= +\infty\qquad \text{ for a.e. }x\in F\setminus E,
\\
		\end{split}
	\end{equation}
\begin{equation}\label{eq:boh}
\begin{aligned}
& |E|\geq \frac{n^n \omega_n|F|^\dimension}{{2^n}(2
\constlemmabase
+\Hausdorff(\partial^*F))^\dimension}>0, 
\\
& \Hausdorff(F\cap\partial^*E)=\Hausdorff(F\cap\partial^* (F\setminus E))\le 
\constlemmabase.
\end{aligned}
\end{equation}
If furthermore
\begin{equation}\label{eq:if_furthermore}
|\varphi_k(x)-2\pi d_k|\rightarrow +\infty \qquad {\rm for~ a.e.~}
x\in \Om\setminus E,
\end{equation}
then 
\begin{align}\label{eq:lem5.9:bis}
 \liminf_{k\rightarrow +\infty}\Hausdorff\big(
S_{\varphi_k}\cap \sottoaperto\big)\ge \Hausdorff\big((
S_{\varphi_\infty} \cup \partial^*E)\cap \sottoaperto\big)
\quad \text{ for any open set } \sottoaperto
\subseteq \Om.  
\end{align}
\end{lemma}
\begin{proof}
Define\footnote{There is no special reason in the choice of $\varphi_1$; any 
element of the sequence $(\varphi_k)$ could be chosen as well.}
	$v_k:=\varphi_k-\varphi_1 \in GSBV^p(\Om)$ which, 
since $\varphi_k$ and $\varphi_1$ are liftings of $u$, is 
piecewise constant (i.e., the absolutely continuous 
and Cantor parts of $D v_k$ vanish)	
and takes values in $2\pi\mathbb Z$. Therefore, it induces a
Caccioppoli partition 
$\{V^k_m:m\in \mathbb Z\}$
of $F$, where 
$$
\text{for any } m\in \mathbb Z \qquad
V^k_m:=\{x\in F:v_k(x)=2\pi m\} \text{ has finite perimeter}, 
$$
$$
v_k = \sum_{m \in \mathbb Z} 2\pi m \chi_{V_m^k}, \qquad
\sum_{m\in\mathbb Z}|V^k_m|=|F|. 
$$
Moreover, since $\partial^* V_m^k \subset S_{\varphi_k} \cup S_{\varphi_1}\cup \partial^*F$,
the perimeter of the partition satisfies
	\begin{equation}\label{eq:partition}
		\begin{split}
\frac12\sum_{m\in \mathbb Z}P(V^k_m)&\leq \mathcal H^\codone(S_{\varphi_k}\cup S_{\varphi_1})+\mathcal H^\codone(\partial^*F)\\
			&\leq\mathcal H^\codone(S_{\varphi_k})+
\mathcal H^\codone( S_{\varphi_1})+\mathcal H^\codone(\partial^*F) 
\\
&\leq 
2\constlemmabase
+\mathcal H^\codone(\partial^*F)\,,
		\end{split}
	\end{equation} 
where $C$ is the constant in \eqref{CI}.
Now, for any $k \in \mathbb N$, $k \geq 1$,  we select
$d_k\in\mathbb{Z}$ for which 
 $$
|V^k_{d_k}| 
=
\max \left\{|V^k_m| : m \in \mathbb Z\right\}.
$$
In particular, from \eqref{eq:partition},
\begin{equation}\label{eq:equibounded_perimeters}
P(V^k_{d_k})\leq2(2\constlemmabase
+\mathcal H^\codone(\partial^*F)). 
\end{equation}
Using also the isoperimetric inequality, we find
\begin{equation*}	
\begin{aligned}
		|F|
=&
\sum_{m\in\mathbb Z}|V^k_m|\leq |V^k_{d_k}|^{\frac{1}{\dimension}}
\sum_{m\in\mathbb Z}|V^k_m|^{1-\frac{1}{\dimension}}
\leq 
\frac{|V^k_{d_k}|^{\frac{1}{\dimension}}}{\dimension\omega_n^{1/\dimension}}
\sum_{m\in\mathbb Z}
P(V^k_m)
\\
\leq &
\frac{|V^k_{d_k}|^{1/\dimension}}{n\omega_n^{1/n}} 
2(2\constlemmabase
+\mathcal H^\codone(\partial^*F)).
	\end{aligned}
\end{equation*}
Thus passing to the limit
\begin{equation}\label{eq:equibounded_volumes}
\liminf_{k\rightarrow +\infty}|V^k_{d_k}|\geq \frac{\dimension^\dimension \omega_\dimension|F|^\dimension}{2^n(2\constlemmabase
+\mathcal H^{n-1}(\partial^*F))^\dimension}.
\end{equation}
 
Hence, from \eqref{eq:equibounded_perimeters} and
\eqref{eq:equibounded_volumes},
 there are a not-relabelled subsequence 
and a finite perimeter set 
$$
\finperset
\subset F
$$
 such that 
	\begin{equation}\label{eq:conv_characteristic_functions}
		(\chi^{}_{V^k_{d_k}})
{\rm ~converges~to}~ \chi_{\finperset} \ \text{ a.e. in } F \text{ and weakly star in }BV(\Omega;\{0,1\}) \text{ as } k \to +\infty.
	\end{equation}
	Consequently, 
\begin{equation}\label{eq:stima_volume_E1_dal_basso}
|\finperset|\geq \frac{\dimension^\dimension \omega_\dimension|F|^\dimension}{2^n(2\constlemmabase
+\mathcal H^\codone(\partial^* F))^\dimension}.
\end{equation}
Next, we  define
	$$
\widehat \varphi_k:=\varphi_k-2\pi d_k.
$$
 Notice that $\mathcal{H}^\codone(S_{\widehat{\varphi}_k})
\le \mathcal{H}^\codone(S_{{\varphi}_k})\le \constlemmabase$;
	moreover 
	by construction
	$$\widehat \varphi_k=\varphi_1\text{ on }V^k_{d_k}
\qquad \forall
k \in \mathbb N,
$$
	and, from 
\eqref{eq:conv_characteristic_functions}, 
	\begin{align*}
		\widehat \varphi_k\rightarrow \varphi_1\qquad \text{pointwise a.e. in }\finperset\,.
	\end{align*}
However, there could be other regions of $F$ out of $\finperset$,
where the 
sequence $(\widehat \varphi_k)$ 
converges pointwise,
that we need to identify. To this aim,  
let us consider, for all integers $N\ge1$, the sequence 
of truncated functions $
(\widehat \varphi_k\wedge N)\vee (-N)
$, which 
is precompact in $SBV^p(\Omega)$.
 Using a diagonal argument, we 
 select a further (not-relabelled) subsequence such that $(\widehat \varphi_k\wedge N)\vee (-N)$ 
converge in $L^1(\Om)$ and pointwise a.e. in $F$ for all $N \in \mathbb N$.
Define 
$$
F_N := \left\{
x \in F : \lim_{k\rightarrow \infty}\vert (\widehat \varphi_k(x)\wedge N)\vee (-N)
\vert = N
\right\}.
$$ 
Then $F_N \supset F_{N+1}$, and 
$$
\bigcap_{N \in \mathbb N} 
F_N 
:= \{x \in F : \vert \widehat \varphi_k(x)
\vert \to +\infty \}.
$$
	 As a consequence, setting 
\begin{equation}\label{eq:E_contiene_E1}
E := F \setminus \bigcap_{N} 
F_N,
\end{equation}
we have 
\begin{equation}\label{eq:catena_di_inclusioni}
\finperset \subseteq E \subseteq F,
\end{equation}
$E$ has finite perimeter in $\Om$\footnote{$E$  has finite perimeter in $F$, 
and $F$ has finite perimeter in $\Om$, therefore $E$ has finite perimeter in 
$\Om$.}, 
and 
setting  
$$
 \varphi_\infty(x)=
\begin{cases}
\lim_k\widehat{\varphi}_k(x) & {\rm  in}~  E
\\
0 & {\rm in}~ \Omega\setminus E
\end{cases}\quad\text{for a. e. }x\in\Omega,
$$
we have $ \varphi_\infty\in GSBV^p(\Omega)$,
 $\varphi_\infty=\varphi_1$ a.e. in $G$,
and $\varphi_\infty$ is a lifting of $u$ in $E$. 
	At the same time 
	$$|\widehat \varphi_k(x)|\rightarrow +\infty\qquad \text{for a.e. }x\in F\setminus E.$$ 
Thus, using also 
\eqref{eq:stima_volume_E1_dal_basso} and \eqref{eq:catena_di_inclusioni}, 
which imply
$|E|\geq |\finperset|\geq  \frac{\dimension^\dimension \omega_\dimension|F|^\dimension}{2^n(2\constlemmabase
+\Hausdorff(\partial^*F))^\dimension}$,
statements \eqref{eq:phi_infty_is_a_lifting}-\eqref{eq:boh} are proven.

It remains to prove \eqref{eq:lem5.9:bis}. Assume  $|\widehat{\varphi}_k|\to+\infty$ a.e. in $\Omega\setminus E$. Then, 
by Theorem \ref{teo:vito}, for any subsequence $(\widehat{\varphi}_{k_h})$ we can extract a further subsequence $(\widehat{\varphi}_{k_{h_j}})$ such that 
	\begin{align*}
	\Hausdorff((
	S_{ \varphi_\infty} \cup 
	\partial^*E )\cap \sottoaperto) 
	\leq 
	\liminf_{j\rightarrow +\infty}
	\Hausdorff(S_{\varphi_{k_{h_j}}}\cap \sottoaperto)
	\text{ for any open set  } \sottoaperto\subseteq \Om.
\end{align*}
Hence, the same holds for the original sequence $(\widehat\varphi_k)$,
	\begin{align*}
\Hausdorff((
S_{ \varphi_\infty} \cup 
\partial^*E )\cap \sottoaperto) 
\leq 
		\liminf_{k\rightarrow +\infty}
\Hausdorff(S_{\varphi_k}\cap \sottoaperto)
\text{ for any open set  } \sottoaperto\subseteq \Om,
	\end{align*}
which shows  \eqref{eq:lem5.9:bis}.
This additionally implies 
$$
P(E,F)\leq \constlemmabase= \sup_{k \in \NN}
\Hausdorff(S_{\varphi_k}),
$$
which shows the last equality in \eqref{eq:boh}.
\end{proof}

		\begin{lemma}\label{lem:ricoprimenti}
			Let  $p>1$ and $(\varphi_k)_{k\ge1}\subset 
GSBV^p(\Om)$  be a sequence of functions with 
\begin{equation}\label{eq:costante_ricoprimenti}
\constlemmasuccessivo
:=\sup_{k \in \NN}\Hausdorff(S_{\varphi_k})<+\infty.
\end{equation}
Let $N\ge1$ be an integer, $E_1,\dots,E_N\subset \Om$ be 
pairwise disjoint nonempty finite perimeter sets and $\varphi_\infty^1,\dots,
\varphi_\infty^N$ be functions in $GSBV^p(\Om)$ with 
the following properties: for any $i=1,\dots,N$, 
			\begin{equation}\label{eq:52}
				\begin{split}
					&\varphi_\infty^{(i)}=0\qquad 
\text{ a.e. in }\Om\setminus E_i\,,\\[1em]
					&  
\liminf_{k\rightarrow +\infty}\Hausdorff(S_{\varphi_k}\cap B)\ge \Hausdorff\Big((
S_{\varphi_\infty^{\pip}}\cup 
\partial^*E_i)
\cap B\Big)
  \text{ for any open ball } B\subset\Omega.
				\end{split}
			\end{equation}
Define 
$\Phi_N\in GSBV^p(\Om)$ as
			\begin{align*}
\Phi_N(x)
:=\begin{cases}
\varphi_\infty^{\pip}(x)&\text{if }x\in E_i\text{ for some }i=1,\dots,N\,,\\[1em]
					0&\text{if }x\in \Omega\setminus 
\big(\cup_{i=1}^NE_i\big).
				\end{cases}
			\end{align*}
			Then
\begin{equation}\label{70thesis}
 \liminf_{k\rightarrow +\infty}\Hausdorff(S_{\varphi_k})\ge   \Hausdorff\Big(S_{\Phi_N}\cup\big(\Om \cap 
 \partial^*\big(\cup_{i=1}^N E_i\big)\big)\Big),
\end{equation}
and
\begin{equation}\label{eq:bound_on_reduced_boundaries}
				\Hausdorff\Big(
\Om\cap 
\partial^* \big(
\cup_{i=1}^N E_i
\big)\Big)=\Hausdorff\Big(
\Om\cap
\partial^* \big(\Omega\setminus \cup_{i=1}^N E_i
\big)\Big)\leq \constlemmasuccessivo.
\end{equation}
			
		\end{lemma}
		
	\begin{proof}
The equality in \eqref{eq:bound_on_reduced_boundaries} follows 
from $\Omega\cap\partial^*E=\Omega\cap\partial^*
(\Omega\setminus E)$, and the 
inequality is a consequence of \eqref{70thesis} and \eqref{eq:costante_ricoprimenti}.
So, let us prove \eqref{70thesis}.
	
	To shortcut the notation we set, for all $i=1,\dots,N$,
		\begin{align*}
			\Sigma_i:=
\Om \cap \partial^*E_i, 
\qquad  S_i:=S_{\varphi^{\pip}_\infty}\setminus \partial^*E_i\,, 
		\end{align*}
		and 
		\begin{equation*}
			\Sigma:=\bigcup_{i=1}^N \Sigma_i\,,
			\qquad S:=\bigcup_{i=1}^N S_i,
		\end{equation*}
are $(\onecod)$-rectifiable with 
$\Hausdorff
(\Sigma)<+\infty$, 
$\Hausdorff(S)<+\infty$.
		We fix $\delta \in (0,1)$ and, for any $i=1,\dots,N$, 
for $\Hausdorff$-a.e. $x\in S_i$ we choose a radius $r(x)$ such that
		\begin{equation}\label{63bis}
			\begin{aligned}
& B_\rho(x)\subseteq\Om\,,
\\
& \Hausdorff(B_\rho(x)\cap S_i)\geq (1-\delta)\omega_{n-1}\rho^\codone,
\\
& \Hausdorff(B_\rho(x)\cap \Sigma)
+
\sum_{\substack{m\neq i\\m=1}}^N\Hausdorff(B_\rho(x)\cap S_m)
\leq \delta \omega_{n-1}\rho^\codone
			\end{aligned}
\qquad \qquad 
\forall \rho\in (0, r(x)).
		\end{equation}
We collect all such balls $B_\rho(x)$ satisfying \eqref{63bis} in a family 
denoted $\mathcal B_i$.
		Furthermore, possibly
reducing the value of $r(x)$, 
for any $i=1,\dots,N$ and for $\Hausdorff$-a.e. $x\in \Sigma_i$ we may suppose that 
		\begin{equation}\label{72bis}
			\begin{split}
&B_\rho(x)\subseteq\Om\,,\\ & \Hausdorff(B_\rho(x)\cap \Sigma_i)\geq (1-\delta)\omega_{n-1}\rho^\codone,
\\
& \Hausdorff(B_\rho(x)\cap S)
+
\sum_{\substack{m\neq i\\m=1}}^N\Hausdorff(B_\rho(x)\cap \Sigma_m)
\leq \delta\omega_{n-1} \rho^\codone,
			\end{split}
\qquad \qquad
\rho\in
(0, r(x)),
\end{equation} 
		and we collect such balls in a 
family $\mathcal B_{i+N}$.
		The family 
$\cup_{n=1}^{2N}\mathcal B_n$, forms a Vitali covering of $\Sigma\cup S$, and so by Vitali covering theorem we can 
choose countable many points $x_k \in \Om$ and radii $\rho_k\in (0, r(x_k))$ such that the 
family 
$$
\mathcal B:=\left\{B\in \bigcup_{\indice
=1}^{2N}\mathcal B_\indice:B=B_{\rho_k}(x_k) \text{ for some }k\in \mathbb N
\right\}
$$
 consists of mutually disjoint balls and covers 
 $\Sigma\cup S$
up to a $\Hausdorff$-negligible set, 
with 
$$
\sum_{k=1}^{+\infty}\rho_k^\codone < +\infty.
$$

In addition
\begin{align}\label{eq:73}
			&\Hausdorff(S\cup \Sigma)=
\sum_{B\in \mathcal B}\Hausdorff(B\cap (S\cup \Sigma))\geq (1-\delta)\omega_{n-1}
\sum_{k=1}^{+\infty}\rho_k^\codone,
		\end{align}
		the inequality following from \eqref{63bis} and \eqref{72bis}. 
		From the same 
formulas, for any $n=1,\dots,N$, 
		and  $B=B_\rho(x)\in\mathcal{B}_i$, $i\ne 
\indice$, it holds
		\begin{align*}
			\Hausdorff(B\cap S_\indice
)\le \sum_{\substack{m\ne i\\m=1}}^{N}
			\Hausdorff(B\cap S_m)\le \delta\omega_{n-1} \rho^\codone.
		\end{align*}
		This, together with \eqref{eq:73}, imply
		\begin{equation}\label{eq:stella}
			\sum_{B\in \mathcal B\setminus\mathcal B_\indice }
\Hausdorff(B\cap S_\indice)\leq \delta\omega_{n-1}\sum_{k=1}^{+\infty}
{\rho_k^\codone}\leq 
\frac{\delta}{1-\delta}\Hausdorff(S\cup \Sigma).
		\end{equation}
		Similarly, for all $h=1,\dots,N$,
		\begin{equation}\label{eq:pila}
			\sum_{B\in \mathcal B\setminus\mathcal B_{N+h} }
\Hausdorff
(B\cap \Sigma_h)\leq \frac{\delta}
{1-\delta}\Hausdorff(S\cup \Sigma).
		\end{equation}
		From \eqref{eq:stella} and \eqref{eq:pila}
it follows 
		\begin{align}\label{key66bis}
			\sum_{B\in \mathcal B\cap \mathcal B_\indice }
\Hausdorff
(B\cap S_\indice)&=	\sum_{B\in \mathcal B}\Hausdorff
(B\cap S_\indice)-
			\sum_{B\in \mathcal B\setminus \mathcal B_\indice }
\Hausdorff
(B\cap S_\indice)\nonumber\\
			&\geq \Hausdorff
(S_\indice)- \frac{\delta \Hausdorff
(S\cup \Sigma)}{1-\delta}\,,
		\end{align}
		for all $\indice=1,\dots,N$, and 
		analogously for all $h=1,\dots,N$,
		\begin{align}\label{key67bis}
			\sum_{B\in \mathcal B\cap \mathcal B_{N+h} }
\Hausdorff
(B\cap \Sigma_h)\geq 
\Hausdorff
(\Sigma_h)- \frac{\delta \Hausdorff
(S\cup \Sigma)}{1-\delta}.
		\end{align}
		From \eqref{eq:52}, for any $\indice=1,\dots,N$ we obtain
		\begin{align}\label{S_nbis}
			\liminf_{k\rightarrow +\infty}
\Hausdorff
(S_{\varphi_k}\cap B)\ge \Hausdorff
(S_{\varphi^\indice_\infty}\cap B)\ge	
\Hausdorff
(S_\indice\cap B)\,,
		\end{align}
		and for all $h=1,\dots,N$
		\begin{align}\label{Sigma_nbis}
			\liminf_{k\rightarrow +\infty}
\Hausdorff
(S_{\varphi_k}\cap B)\ge 
			\Hausdorff
(B \cap \partial^*E_h)\ge 	\Hausdorff
(\Sigma_h\cap B) \,.
		\end{align}
		Now, summing \eqref{S_nbis} over all $B\in \mathcal B\cap \mathcal B_\indice$ and \eqref{Sigma_nbis} over all $B\in \mathcal B\cap \mathcal B_{N+h}$, and then over $n,h$ respectively, we infer
\begin{equation}\label{eq:double_sum}
\begin{aligned}
& \sum_{\indice=1}^N\sum_{B\in \mathcal B\cap \mathcal B_\indice }
\Hausdorff
(S_\indice\cap B)+\sum_{h=1}^N\sum_{B\in \mathcal B\cap \mathcal B_{N+h} }
\Hausdorff
(\Sigma_h\cap B)
\\
\leq & \liminf_{k\to +\infty} \sum_{\indice=1}^{2N}\sum_{B\in \mathcal B\cap \mathcal B_\indice }\Hausdorff
(S_{\varphi_k}\cap B)\leq \liminf_{k\rightarrow +\infty}
\Hausdorff
(S_{\varphi_k}),
\end{aligned}		
\end{equation}
		where the penultimate inequality is a 
consequence of Fatou's Lemma.
Combining \eqref{eq:double_sum} with \eqref{key66bis} and \eqref{key67bis} we get
		\begin{align*}
			\sum_{\indice=1}^N
\Hausdorff
(S_\indice)+\sum_{h=1}^N
\Hausdorff
(\Sigma_h) - \frac{2\delta N 
\Hausdorff
(S\cup \Sigma)}{1-\delta}\leq \liminf_{k\rightarrow +\infty}
\Hausdorff
(S_{\varphi_k})\,.
		\end{align*}
		By the arbitrariness of $\delta>0$ we conclude
		$$\Hausdorff
(S\cup\Sigma)\leq \liminf_{k\rightarrow +\infty}\Hausdorff
(S_{\varphi_k})\,,$$ that implies \eqref{70thesis}.
\end{proof}

\subsection{Proof of Theorem \ref{thm:compactness&lsc}}
		
Let $M :=\sup_{k\in\mathbb N} \Hausdorff(S_{\varphi_k}) < +\infty$.
To show compactness of 
$(\varphi_k)$, we utilize an iterative argument.  
\\
\noindent \textbf{Base case $N=1$.}
We set 
\begin{equation}\label{eq:F_1}
F_1:=\Omega.
\end{equation}
 From Lemma \ref{lem:passo_base} we 
find, for  a not-relabelled subsequence, 
a finite perimeter set $E_1\subseteq F_1$, 
a sequence $(d^{(1)}_k)_{k\ge1}\subset\mathbb Z$  and a function 
$\varphi^{(1)}_\infty\in GSBV^p(\Om)$,
 such that 
\begin{equation*}
	\begin{aligned}
&\varphi^{(1)}_\infty \text{ 
is a lifting of } u \text{ in } E_1,
\qquad	\varphi^{(1)}_\infty=0\quad\text{ in }F_1\setminus E_1\,,\quad 
\\
&			|E_1|
\geq 
\frac{\dimension^\dimension \omega_\dimension
|F_1|^\dimension}{2^n(2M+\Hausdorff(\partial\Om))^\dimension}\,,
\end{aligned}
\end{equation*} 
and
\begin{equation}
\begin{aligned}
& (\varphi_k(x)-2\pi d^{(1)}_k)\rightarrow \varphi^{(1)}_\infty(x)
\qquad \text{ for a.e. }x\in E_1\,,\\
&  	|\varphi_k(x)-2\pi d^{(1)}_k|\rightarrow +\infty\qquad 
\text{ for a.e. }x\in F_1\setminus E_1.
\end{aligned}
\end{equation}
Moreover, since $F_1 = \Om$, from the final part of
the statement of Lemma \ref{lem:passo_base}, we have 
\begin{equation}\label{eq:final_part}
\liminf_{k\rightarrow +\infty}\Hausdorff
(S_{\varphi_k}\cap \sottoaperto)\geq \Hausdorff((
S_{\varphi^{(1)}_\infty} \cup 
\partial^*E_1
)\cap \sottoaperto)
\text{ for any open set } \sottoaperto\subseteq F_1,
		\end{equation}
see \eqref{eq:lem5.9:bis}. 
		
\noindent \textbf{Iterative case $N\leadsto N+1$.}  
		Let $N\ge2$ be an integer,
$E_1,\dots,E_{N}\subset\Omega$ be 
pairwise
disjoint nonempty finite perimeter sets,
and define, together with \eqref{eq:F_1},
$$
F_i:=\Omega\setminus \bigcup_{j=1}^{i-1}E_j \ \ \text{ for } i=2,\dots,N,
$$
so that 
$$
E_1 \subseteq F_1, \ E_2 \subseteq F_2, \dots,
\  E_N \subseteq F_N.
$$
Suppose
that:
\begin{itemize}
\item[(i)] For all $i=1,\dots,N$,  
there exists a function $\varphi_\infty^{{(i)}}\in GSBV^p(\Omega)$ 
which is 
a lifting of $u$ in  $E_i$, and 
$\varphi_\infty^{{(i)}}=0$  in $\Om\setminus E_i$;
\item[(ii)] For all $i=1,\dots,N$ we have 
\begin{equation*}
|E_i|\geq \frac{
\dimension^\dimension \omega_\dimension
|F_i|^\dimension}{2^n(2M+\Hausdorff(\partial^*F_i))^\dimension}\geq\frac{
\dimension^\dimension \omega_\dimension
|F_i|^\dimension}{2^n(3M+\Hausdorff(\partial\Om))^\dimension};
\end{equation*} 
\item[(iii)]  
For all $i=1,\dots,N$, there exist 
 sequences
$(d_k^{(i)})_{k}\subset \mathbb{Z}$, 
such that 
\begin{equation}
\begin{split}\label{eq:conv-m-insiemi}
&(\varphi_k(x)-2\pi \dki
)\rightarrow \varphi^{(i)}_\infty(x)\qquad \text{ for a.e. }x\in E_i\,,\\
					&|\varphi_k(x)-2\pi 
\dki
|\rightarrow +\infty\qquad \text{ for a.e. }x\in \Om\setminus E_i\,,\\
					& \liminf_{k\rightarrow +\infty}
\Hausdorff(S_{\varphi_k}\cap \sottoaperto)
\geq 
\Hausdorff((S_{\varphi_\infty^{(i)}} \cup \partial^*E_i)\cap \sottoaperto)
\text{ for any open set } \sottoaperto\subseteq\Omega.
				\end{split}
\end{equation}
\end{itemize}
		
We now want to 
find a finite perimeter set $E_{N+1} \subseteq \Om$ 
disjoint from $\cup_{i=1}^N E_i$  
and a function $\varphi_\infty^{(N+1)} \in GSBV^p(\Om)$, such that, for a not-relabelled subsequence,  
$E_1,\dots, E_{N+1}$, $(\varphi_k)$, and $\varphi_\infty^{(N+1)}$ 
satisfy properties (i)-(iii) above,  
with $N$  replaced by $N+1$.
		
To this purpose we set, for $N\ge1$,
$$
F_{N+1}:=\Om\setminus
\Big(\bigcup_{i=1}^{N}E_i\Big), 
$$
and 
{we let $\Phi_N\in GSBV^p(\Om)$ be defined as
			\begin{equation}\label{def:phi-barrato-m}
\Phi_N(x):=\begin{cases}
					\varphi_\infty^{\pip}(x)&\text{if }x\in E_i\text{ for }i=1,\dots,N\,,\\
					0&\text{if }x\in 
\Omega.
				\end{cases}
			\end{equation}
If $F_{N+1}=\emptyset$ there is nothing to prove. Assume then $F_{N+1}\not=\emptyset$.
 From Lemma \ref{lem:ricoprimenti} we have
			\begin{equation}\label{eq:liminf-n-insiemi}
\liminf_{k\rightarrow +\infty}\Hausdorff
(S_{\varphi_k})\geq 
\Hausdorff ( S_{\Phi_N}\cup(\Omega\cap\partial^*(\cup_{i=1}^NE_i)))\,,
			\end{equation}
			and 
			\begin{equation}\label{eq:controllo-perimetro-Fm+1}
\Hausdorff\Big(\Omega\cap\partial^*(\cup_{i=1}^NE_i)\Big)= 
\Hausdorff(\Omega \cap\partial^* F_{N+1})\le M\,.
			\end{equation}
		}
		Next, 
applying {Lemma \ref{lem:passo_base}} to $F=F_{N+1}$, 
there exist a sequence $(d_k^{\pNpop})_{k\ge1} \subset \mathbb Z$, a 
finite perimeter set $E_{N+1}\subseteq F_{N+1}$ (and thus
$E_{N+1} \cap (\cap_{i=1}^N E_i) = \emptyset$) and a function 
$\varphi_\infty^{\pNpop}\in GSBV^p(\Omega)$ such that  \begin{equation*}
			\varphi^{\pNpop}_\infty=0\quad\text{ in }\Omega\setminus E_{N+1}\,,\quad \varphi_\infty^{\pNpop}\text{ is a lifting of } u \text{ in } E_{N+1}\,,
		\end{equation*}
		\begin{equation}\label{eq:controllo-area-Fm+1}
			|E_{N+1}|\geq \frac{\dimension^\dimension
\omega_\dimension|F_{N+1}|^2}{2^n(2M+\Hausdorff(\partial^*F_{N+1}))^2}\,,
		\end{equation}
		and 
		\begin{equation}\label{eq:lem5.9BIS}
			\begin{split}
				&(\varphi_k(x)-2\pi d_k^{\pNpop})\rightarrow \varphi^{\pNpop}_\infty(x)\qquad \text{ for a.e. }x\in E_{N+1}\,,\\
				&  	|\varphi_k(x)-2\pi d^{\pNpop}_k|\rightarrow +\infty\qquad \text{ for a.e. }x\in F_{N+1}\setminus E_{N+1}\,,\\
				&\liminf_{k\rightarrow +\infty}
\mathcal H^1(S_{\varphi_k}\cap \sottoaperto)
\geq \mathcal H^1((
S_{\varphi_\infty^{\pNpop}}
\cup
\partial^*E_{N+1} 
)\cap \sottoaperto)
\text{ for any open set } \sottoaperto
\subseteq F_{N+1}. 
			\end{split}
		\end{equation}
{Combining \eqref{eq:controllo-perimetro-Fm+1} and \eqref{eq:controllo-area-Fm+1} we readily get
			\begin{equation*}
				|E_{N+1}|\geq \frac{\dimension^\dimension
\omega_\dimension|F_{N+1}|^2}{2^n(2M+\Hausdorff(\partial^*F_{N+1}))^2}\geq\frac{\dimension^\dimension \omega_\dimension|F_{N+1}|^2}{2^n(3M+\Hausdorff(\partial\Om))^2} >0\,.
			\end{equation*}
		}
		Moreover by \eqref{eq:conv-m-insiemi}, also
		\begin{align}
			|\varphi_k(x)-2\pi d^{\pNpop}_k|\rightarrow +\infty\qquad \text{ for a.e. }x\in \Om\setminus E_{N+1}\,.
		\end{align}
		Thus 
the sets $E_1,\dots, E_{N+1}$ satisfy 
(i)-(iii) with $(d_k^{\pip})_{k\ge1}\subset\mathbb{Z}$ and $\varphi_\infty^{\pip} \in GSBV^p(\Omega)$ for $i\in\{1,\dots,N+1\}$.\\
		
		\noindent \textbf{Conclusion.} 
We now combine the base case and the iterative case to conclude the proof.  	
		{First we note that each time we apply the iterative case we have to extract a subsequence. Hence, taking a diagonal (not-relabelled) subsequence of $(\varphi_k)_{k\ge1}$, this yields
 a sequence $(E_i)_{i\ge1}$ of mutually disjoint finite perimeter sets in $\Omega$ such that for every $m\ge1$, $E_1,\dots,E_m$ satisfy 
properties (i)-(iii) (with $\indiceteo$ in place of $N$) with the 
corresponding $(d^{\pip}_k)_{k\ge1} \subset \mathbb Z$ and $\varphi_\infty^{\pip}\in 
GSBV^p(\Omega)$ for $i\in \{1,\dots,\indiceteo\}$.
In particular, from \eqref{eq:liminf-n-insiemi},
\begin{equation}\label{liminf-passo-n}
				\liminf_{k\rightarrow +\infty}
\Hausdorff(S_{\varphi_k})\geq \Hausdorff( 
S_{\Phi_\indiceteo}\cup(\Omega\cap\partial^*(\cup_{i=1}^
\indiceteo E_i)))\,,\quad\forall \indiceteo
\ge1
			\end{equation}
			with $\Phi_\indiceteo$ 
as in \eqref{def:phi-barrato-m} for $N=\indiceteo$.
		}
		We next show that $$|\Om\setminus (\cup_{i=1}^{+\infty} E_i)|=0.$$ To this aim, since $\sum_{\indiceteo=1}^{+\infty}
|E_\indiceteo|<+\infty$, the sequence 
$(|E_\indiceteo|)$ tends to zero as $\indiceteo
\rightarrow +\infty$\footnote{Note that it might be $E_i=\emptyset$ for $i\ge\bar i$, this case is simpler to treat.}, and by the inequality 
$|E_\indiceteo|\geq \frac{\dimension^\dimension \omega_\dimension
|F_\indiceteo|^2}{2^\dimension(3M+\Hausdorff(\partial \Om))^2}$  we also 
infer that $|F_\indiceteo|\rightarrow 0$. 
In particular, $$|\Om\setminus (\cup_{i=1}^{+\infty} E_i)|=\lim_{\indiceteo
\rightarrow+\infty}|\Om\setminus (\cup_{i=1}^{\indiceteo
-1} E_i)|=\lim_{\indiceteo
\rightarrow +\infty}|F_\indiceteo|=0\,.$$ Finally we define
		$$\varphi_\infty(x):=\varphi_\infty^{\pip}(x)\qquad \text{ 
if }x\in E_i {\rm ~for ~some~} i \in \NN.$$
		We now show that $\varphi_\infty\in GSBV^p(\Om)$ 
and that \eqref{eq:lsc} holds true.
			By definition of $\varphi_\infty$ 
and $\Phi_\indiceteo$ it follows that  
			\begin{equation*}
				\Hausdorff(S_{\varphi_\infty})
				\le 
				\lim_{m\to +\infty}\Hausdorff( S_{\Phi_\indiceteo})
				\le 
\lim_{m\to +\infty}\Hausdorff( S_{\Phi_\indiceteo}
\cup(\Omega\cap\partial^*(\cup_{i=1}^\indiceteo
E_i))).
			\end{equation*}
			This together with \eqref{liminf-passo-n} 
yields \eqref{eq:lsc}. 
			Eventually, being $\varphi_\infty$ limit of liftings 
of $u$, it is itself a lifting of $u$ in $\Omega$ and therefore from 
the identity $|\nabla u|=|\nabla \varphi_\infty|$ 
we infer $\varphi_\infty\in GSBV^p(\Omega).$
\qed

\subsection{Proof of Theorem \ref{teo:compactness}}
\label{sec:compactness_of_liftings_of_a_converging_sequence}
	Let $\varphi\in SBV^p(\Om)\cap L^\infty(\Om)$ be a lifting of $u$ with 
\begin{equation}\label{eq:varphi_BV}
|\varphi|_{BV}\leq 2|u|_{BV}\leq C,
\end{equation}
whose existence is ensured by Theorem \ref{thm:dav-ign} and Remark \ref{rem:dav-ign}. 
The idea of the proof is to 
construct a sequence $(\widetilde\psi_k)_{k\ge1} \subset SBV^p(\Om) \cap L^\infty(\Om)$
(see \eqref{eq:w_k}) 
having the following properties: 
			\begin{equation}\label{eq:cons}
				\begin{split}
	&\widetilde	\psi_k\to0 \quad {\rm ~in } ~ L^1(\Omega) 
	{\rm ~and~
		weakly}^* ~ {\rm in~} BV(\Omega),\\
&	\widetilde\varphi_k:=	\varphi_k-\widetilde\psi_k= \varphi+2\pi\sum_{z\in \mathbb Z}z\chi_{F_k^z}\,,\\
&	\sup_{k \in \NN}
	\Hausdorff
	(\widetilde\varphi_k)<+\infty,
				\end{split}
\end{equation}
			with $(F_k^z)_{z\in\mathbb{Z}}$ a 
Caccioppoli partition of $\Omega$ for all $k\ge1$,
see also \eqref{eq:unse}. Note that
each $\widetilde \varphi_k$ is a lifting of the limit map $u$.
Hence we can apply Theorem \ref{thm:compactness&lsc} 
to the sequence $(\widetilde \varphi_k)$, 
and eventually from the convergence of (a not relabelled 
subsequence of) $(\widetilde\varphi_k)_{k\ge1}$ and Theorem \ref{teo:vito}	deduce compactness and lower semicontinuity for the original sequence $(\varphi_k)_{k\ge1}$. 
We now give the details of the proof.
\medskip

		\noindent	\textit{Step 1: construction of $\widetilde\psi_k$.}
		For $k\geq 1$ 
define first
$$
\psi_k:=\varphi_k-\varphi \in GSBV^p(\Om). 
$$
Since $\varphi$ is a lifting of $u$,
we have $\jump{\varphi} \in  2\pi \mathbb Z$ on $S_\varphi\setminus S_u$, 
thus, using also \eqref{eq:varphi_BV},  
$\Hausdorff(S_\varphi)=
\Hausdorff
(S_u)+\Hausdorff(S_\varphi\setminus S_u)<+\infty$,
and $|\nabla u|=|\nabla \varphi|$ 
a.e. on $\Om$.
Moreover, since $S_{\psi_k}\subseteq S_{\varphi_k}\cup S_{\varphi}$ and 
$|\nabla \psi_k|\le |\nabla \varphi_k|+|\nabla \varphi|=|\nabla u_k|+|\nabla u|$  
a.e. on $\Om$, we have, using also \eqref{eq:dododo},
		\begin{equation}\label{bound-vk}
			\sup_{k\geq 1} \Big(\int_\Om|\nabla\psi_k|\dx+
\Hausdorff
(S_{\widetilde\psi_k})\Big)\le C<+\infty\,.
		\end{equation}
For every $k \in \NN$ and $z\in \mathbb{Z}$ we define
$$\psi_{k}^z:=(\psi_k\wedge 2\pi z)\vee 2\pi(z-1)\in SBV^p(\Om)\,,$$ whose jump set decomposes as
		\begin{align*}
			S_{\psi_k^z}=S^1_{\psi_k^z}\cup S^2_{\psi_k^z}\cup S^3_{\psi_k^z}\cup S_{\psi_k^z}^4\,,
		\end{align*}
		with
		\begin{align*}
			&S^1_{\psi_k^z}:=\{x\in S_{\psi_k^z}:2\pi(z-1)<(\psi_k^{z})^-(x)<(\psi_k^z)^+(x)<2\pi z\}\,,\\
			&S^2_{\psi_k^z}:=\{x\in S_{\psi_k^z}:2\pi(z-1)=(\psi_k^{z})^-(x)<(\psi_k^z)^+(x)<2\pi z\}\,,\\
			&S^3_{\psi_k^z}:=\{x\in S_{\psi_k^z}:2\pi(z-1)<(\psi_k^{z})^-(x)<(\psi_k^z)^+(x)=2\pi z\}\,,\\
			&S^4_{\psi_k^z}:=\{x\in S_{\psi_k^z}:2\pi(z-1)=(\psi_k^{z})^-(x)<(\psi_k^z)^+(x)=2\pi z\},
		\end{align*}
	where $(\psi_k^z)^\pm$ 
are
the two traces of $\psi_k^z$ on $S_{\psi_k^z}$.
		As $\psi_k\in GSBV^p(\Om)$  we have,  
up to a $\Hausdorff$-negligible set,
		\begin{equation}\label{eq:union_of_three_jumps}
			\bigcup_{z\in \mathbb Z}S^1_{v_k^z}\cup \bigcup_{z\in \mathbb Z}S^2_{v_k^z}\cup \bigcup_{z\in \mathbb Z}S^3_{v_k^z}\subseteq S_{v_k}\,.
		\end{equation}
Notice that,
for $\Hausdorff$-a.e. $x\in S_{\psi_k}$, 
$x$ belongs to at most $2$ sets appearing in 
the union on the 
left-hand side of \eqref{eq:union_of_three_jumps}; hence in particular from \eqref{bound-vk}, for all $k\ge1$
		\begin{align}\label{sum_1}
\sum_{z\in \mathbb Z}\left(\Hausdorff(S^1_{\psi_k^z})+
\Hausdorff
(S^2_{\psi_k^z})+\Hausdorff(S^3_{\psi_k^z})\right)\leq 2\Hausdorff(S_{\psi_k})\le C\,,
		\end{align}for some constant $C>0$. 
		Furthermore by definition of $\psi_k^z$,
		\begin{align}\label{sum_2}
			\int_\Om|\nabla \psi_k|\dx=\sum_{z\in \mathbb Z}\int_\Om|\nabla \psi_k^z|\dx\le C
\qquad \forall k \in \mathbb N.
		\end{align}
Consider now the function 
$\tau_k^z\in SBV^p(\Omega)$, defined by
$$
\tau_k^z:=
\Big|
\psi^z_k-2\pi\Big(z-\frac12\Big)\Big|, 
$$
which satisfies  $0\le\tau_k^z\le\pi$,
$$
\int_\Om|\nabla \psi^z_k|\dx=\int_\Om|\nabla \tau^z_k|\dx,\qquad |\jump{\tau_k^z}|\leq |\jump{\psi_k^z}| \;\;\;\Hausdorff-\text{a.e. on }S_{\tau_k^z}\,,
$$ 
		\begin{align*}
			S_{\tau_k^z}\subseteq S^1_{\psi_k^z}\cup S^2_{\psi_k^z}\cup S^3_{\psi_k^z}\,,
		\end{align*}
		owing to the fact that $\tau^z_k$ 
has null jump on $S^4_{\psi_k^z}$.
		In particular,
we infer from \eqref{sum_1} and \eqref{sum_2} that 
		\begin{equation}\label{eq:vert_D}
			\sum_{z\in \mathbb Z}|D\tau^z_k|(\Om)\le C \qquad \forall
k \in \mathbb N,
		\end{equation} 
		for some positive constant $C$. 

Now, for $t\ge 0$, let  $(E^z_k)^t:=\{x\in \Om:\tau^z_k(x)<t\}$; 
		by the coarea formula, 
		\begin{align*}
			\int_0^{\pi}\Hausdorff
(\partial^*(E_k^z)^t)\dt=|D\tau_k^z|(\Om)\,,
		\end{align*}
		hence, summing over $z$
and using \eqref{eq:vert_D}, 
		\begin{align*}
			\int_0^{\pi}	\sum_{z\in \mathbb Z}(\Hausdorff
(\partial^*(E_k^z)^t))\dt=
			\sum_{z\in \mathbb Z}\int_0^{\pi}\Hausdorff
(\partial^*(E_k^z)^t)\dt\le C.
		\end{align*}
Whence, 		
for any $k\in \mathbb Z$ we can find
a number $t_k\in (0,\pi/2)$ such that 
		\begin{align}
			\sum_{z\in \mathbb Z}\Hausdorff
(\partial^*(E_k^z)^{t_k})\leq C\,,
		\end{align}
		with $C>0$ independent of $k$.
		We observe that 
$$(E_k^z)^{t_k}=
\left\{x\in\Om:2\pi\Big(z-\frac12\Big)-t_k<\psi_k^z(x)=\psi_k(x)<2\pi\Big(z-\frac12\Big)+t_k\right\}\,.$$
		For every $k$ let $(F_k^z)_{z\in\mathbb{Z}}$ be the 
Caccioppoli 
partition of $\Omega$ defined as
		$$F_k^z:=\Big\{x\in\Om:\psi_k(x)\in		\Big( 2\pi\Big(z-\frac12\Big)-t_k,2\pi\Big(z+1-\frac12\Big)-t_k\Big)\Big\}\,.$$
		Thus $\partial^* F_k^z\subset \partial^* (E_k^z)^{t_k}\cup \partial^* (E_k^{z+1})^{t_k}$ and
		\begin{align}\label{55}
			\sum_{{z}\in \mathbb Z}\Hausdorff
(\partial^*F_k^z)\leq { 2}\sum_{{z}\in \mathbb Z}\Hausdorff
(\partial^*(E_k^z)^{t_k})\leq C\,,
		\end{align}
		and so for each $k\in \mathbb N$ the family 
$(F_k^z)_{z\in\mathbb Z}$ is a 
Caccioppoli partition of $\Om$, with equibounded (in $k$) total perimeter. 
		
We finally introduce the map $\widetilde\psi_k\in SBV^p(\Om)\cap L^\infty(\Om)$ as 
\begin{equation}\label{eq:w_k}
\widetilde\psi_k:=\sum_{z\in \mathbb Z}\Big(\psi_k-(2\pi(z-\frac12)-t_k)\Big)\chi_{F_k^z}-\pi-t_k=\psi_k-2\pi\sum_{{z}\in \mathbb Z} z\chi_{F_k^z},
\end{equation}
which satisfies {$-\pi\leq \widetilde\psi_k\leq \pi$}. 
 Moreover, since $S_{\widetilde\psi_k}\subseteq S_{\psi_k}\cup \bigcup_{z\in \mathbb Z} \partial^*F_k^z$, by \eqref{55} we deduce that 
		\begin{align}\label{56}
			\Hausdorff
(S_{\widetilde\psi_k})+\|\widetilde\psi_k\|_{BV}\leq C,
		\end{align}
		for all $k \geq 1$. Therefore, 
up to a not relabelled subsequence, we can also suppose 
\begin{align*}
\widetilde\psi_k\rightarrow \widetilde\psi\qquad \text{ in }L^1(\Om) {\rm ~and}~ 
\text{ weakly}^* {\rm ~in~} BV(\Om) 
		\end{align*}
		for some $\widetilde\psi\in SBV^p(\Om)\cap L^\infty(\Om)$.
As asserted in \eqref{eq:cons}, we now want to show that $\widetilde\psi= 0$ a.e. on $\Om$. 
		Let us go back to the functions $u_k$ that, up to subsequences, are converging pointwise a.e. to $u$.  {Hence for a.e. $x$ we have
			\begin{equation*}
				u_k(x)-u(x)=e^{i\varphi(x)}(e^{i(\varphi_k(x)-\varphi(x))}  -1 )\to0\,.
			\end{equation*}
		}
		This in turn  implies that the map 
$x\mapsto\text{{\rm dist}}(\varphi_k(x)-\varphi(x),2\pi\mathbb Z)
=\text{{\rm dist}}(\psi_k(x),2\pi\mathbb Z)$ 
is converging pointwise a.e. to $0$, and by the
dominated convergence theorem, in $L^1(\Om)$. 
		{Then {there is $\bar k\ge0$ such that} for a.e.  $x\in \Om$ and}
		for all $k\geq {\bar k}$ there is $K=K(x,k)\in \mathbb Z$ such that 
$$
x\in F^{K}_k\quad \text{ and }\quad  |\widetilde\psi_k(x)|=|\psi_k(x)-2\pi K|\rightarrow 0 \qquad \text{as }k\rightarrow +\infty,
$$ 
and so by the dominated convergence theorem we conclude $\widetilde\psi=0$.
		
\medskip
			\textit{Step 2: compactness and 
lower semicontinuity.}
		We  first observe that $\psi_k-\widetilde\psi_k\in GSBV^p(\Om)$
takes values in $2\pi\mathbb Z$. As a 
consequence, the maps $\widetilde \varphi_k
=\varphi_k-\widetilde\psi_k=\psi_k-\widetilde\psi_k+\varphi$ are all liftings of $u$ and  from \eqref{56} 
we also have
		\begin{equation}\label{eq:unse}
			\sup_{k\geq 1}\Hausdorff
(S_{\widetilde \varphi_k})<+\infty.
		\end{equation}
 We now apply Theorem \ref{thm:compactness&lsc} to the sequence $(\widetilde \varphi_k)_{k\ge1}$, 
and deduce the existence of a Caccioppoli partition $(E_i)_{i\ge 1}$ of $\Om$, of
sequences $(\dki)_{k\ge1}\subset 2\pi\mathbb Z$ for all $i\ge1$, 
and of a lifting $\varphi_\infty\in GSBV^p(\Om)$
 of $u$, such that
 \begin{equation*}
\begin{split}
&\lim_{k \to +\infty}	(\widetilde\varphi_k(x)-2\pi \dki
)= \varphi_\infty(x)\quad \text{for a.e. }x\in E_i\,,\\
&\lim_{k \to +\infty} |\widetilde\varphi_k(x)-2\pi 
\dki
|=+ \infty\quad \text{for a.e. }x\in \Omega\setminus E_i\,,
\end{split}
 \end{equation*}
 for all $i\ge1$. 
  As a consequence, by  Step 2, 
  \begin{equation*}
  \begin{split}
  	&\lim_{k \to +\infty}	(\varphi_k(x)-2\pi d^{\pip}_k)=\varphi_\infty(x)\quad \text{for a.e. }x\in E_i\,,\\
  	&\lim_{k \to +\infty}|\varphi_k(x)-2\pi d^{\pip}_k|= + \infty\quad \text{for a.e. }x\in \Omega\setminus E_i\,,
  \end{split}
\end{equation*}
for all $i\ge1$.   
  
It remains to show  \eqref{eq:liminf}. { For all $i\ge1$ let $\varphi_\infty^{\pip}:=\varphi_\infty$ in $E_i$ and $\varphi_\infty^{\pip}:=0$ in $\Omega\setminus E_i$.
  	By Theorem \ref{teo:vito} and the fact that $S_{\varphi_k}=S_{\varphi_k-2\pi d^{\pip}_k}$	it follows
  	\begin{equation*}
  	\liminf_{k\to+\infty}\Hausdorff
(S_{\varphi_k}\cap \sottoaperto)\ge 
  	\Hausdorff
( (S_{\varphi_\infty^{\pip}}\cup\partial^*E_i)\cap \sottoaperto  )\,,
  	\end{equation*}
  for any open set $\sottoaperto\subset\Omega$. Then by Lemma \ref{lem:ricoprimenti} for all $N\ge1$ we get
  \begin{equation*}
  		\liminf_{k\to+\infty}\Hausdorff
(S_{\varphi_k})\ge \Hausdorff
(S_{\overline\varphi_\infty^{\pNp}}\cup(\partial^*(\cup_{i=1}^NE_i)))\,,
  \end{equation*}
  where
  \begin{equation*}
  	\overline\varphi_\infty^{\pNp} :=
  	\begin{cases}
  		\varphi^{\pip}_\infty&\text{ if }x\in E_i\text{ 
for some }i=1,\dots,N\,,\\
  		0&\text{ if }x\in \Omega\setminus(\cup_{i=1}^N E_i)\,.
  	\end{cases}
  \end{equation*}
  Hence,
 observing that $\varphi_\infty=
\overline\varphi_\infty^{\pNp}$ in $\cup_{i= 1}^NE_i$, 
and letting $N\to+ \infty$ we get
   \begin{equation*}
 	\liminf_{k\to+\infty}\Hausdorff
(S_{\varphi_k})\ge \Hausdorff
(S_{\varphi_\infty}\cup(\partial^*(\cup_{i=1}^{+\infty} E_i)))
 	\ge \Hausdorff
(S_{\varphi_\infty})\,.
 \end{equation*}
}
\qed

\section{$\Gamma$-convergence of functionals 
on $\mathbb{S}^1$- valued maps}
\label{sec:Gamma_convergence}
In this section we prove Theorems 
\ref{thm1} and \ref{thm2}. The proof of Theorem \ref{thm1} 
follows by suitably adapting 
the arguments of \cite{AT90,AT92,Focardi}. 
Instead, the proof of Theorem \ref{thm2} (and more 
specifically the lower bound inequality) requires some new ideas which rely on the compactness result 
for liftings (Theorem \ref{teo:compactness}).
For convenience we introduce the localised Modica-Mortola-type (or Allen-Cahn type) functionals
\begin{equation*}
\ModicaMortolaeps
(v, \sottoaperto)
:=\int_{\sottoaperto}\left(\var |\nabla v|^2+\frac{(v-1)^2}{4\var}
\right)\dx \qquad \forall v \in W^{1,2}(\Om),
\end{equation*}
for every open set $\sottoaperto\subseteq\Omega$. 

\subsection{Some density and approximation results for $\Suno$-valued maps}
\label{subsec:some_density}
For $E\subset\R^\dimension$ we denote by $\mathcal{M}^{\codone}(E)$  its $(\codone)$-dimensional Minkowski content, i.e.,
\begin{equation*}
\Minkowski(E)=\lim_{\rho\searrow0}
\frac{|\{x\in\R^\dimension\colon\dist(x,E)<\rho\}|}{2\rho}\,,
\end{equation*}
provided the limit exists.
\begin{prop}[\textbf{Density in $SBV^2(\Omega;\Suno)$}]
\label{prop:density}
Let  $\Omega\subset\R^\dimension$ be a 
bounded open set with Lipschitz boundary,
 $u\in SBV^2(\Omega;\mathbb{S}^1)$ and let 
$\sopraaperto\subset\R^\dimension$ 
be a bounded open set with 
$\Omega\subset\subset\sopraaperto$.  Then $u$ has an extension 
$\extu
\in SBV^2(\sopraaperto;\mathbb{S}^1)$ with 
$$
\Hausdorff(S_\extu\cap\partial\Omega)=0. 
$$
Moreover, there exists a sequence $(\apprextuk)\subset SBV^2(\sopraaperto;\Suno)$ such that:
\begin{enumerate}[label=$(\roman*)$]
	\item $\apprextuk$ converges to $\extu$ in $L^1(\sopraaperto;\Suno)$;
\item $\nabla \apprextuk$ converges to $\nabla \extu$ in 
$L^2(\sopraaperto;\R^{\dimension\times \dimension})$;
\item $\lim_{k\to\infty}\Hausdorff(S_{\apprextuk})=\Hausdorff(S_\extu)$;
\item $\lim_{k\to\infty}
\Hausdorff
(S_{\apprextuk}\cap\overline\Omega)=\Hausdorff
(S_\extu\cap\Omega)=\Hausdorff(S_u)$;
\item $\Hausdorff(\overline S_{\apprextuk}\setminus S_{\apprextuk})=0$ 
and $\Hausdorff(S_{\apprextuk}\cap K)=\mathcal{M}^{\codone}(S_{\apprextuk}
\cap K)$ for every compact set $K\subset\sopraaperto$.
\end{enumerate}
\end{prop}
Note that (ii)-(iii) guarantee the convergence of $\MSSuno(\apprextuk,1)$ to 
$\MSSuno(\extu,1)$ in $\sopraaperto$.

\begin{proof}
	The proof follows by suitably 
adapting \cite[Lemma 4.11]{Braides-Notes} 
and \cite[Proposition 5.3]{AT90}
to the $\Suno$-constrained case. For this reason we give here only the main steps. 
According to 
Remark \ref{rem:dav-ign} there exists a
lifting $\varphi\in SBV^2(\Omega)$ of $u$.  Since $\partial\Omega$ is Lipschitz, by a standard reflection argument we can construct an extension $\widehat\varphi\in SBV^2(\sopraaperto)$ of $\varphi$ such that 
	\begin{align*}
\Hausdorff(S_{\widehat \varphi}\cap\partial\Omega)=0.
	\end{align*}
By defining 
$$
\extu\colon\sopraaperto\to\mathbb{S}^1, \qquad
\extu:=e^{i\widehat\varphi}, 
$$
we immediately get that $\extu$ is
an extension of $u$, and 
	\begin{equation*}
	\int_{\sopraaperto}|\nabla \extu|^2\dx=	\int_{\sopraaperto}|\nabla \widehat\varphi|^2\dx\,, \quad \Hausdorff(S_\extu)\le \Hausdorff(S_{\widehat\varphi})\,,
	\end{equation*}
	and thus $\extu\in SBV^2(\sopraaperto;\mathbb{S}^1)$. 
Moreover \begin{align}\label{6}
\Hausdorff(S_\extu\cap\partial\Omega)\le \Hausdorff(S_{\widehat\varphi}\cap\partial\Omega)=0.	
	\end{align}
	\noindent
	 By \cite[Lemma 4.3]{Carriero-Leaci} we know that, for every integer $k>0$, there exists $\apprextuk\in SBV^2(\sopraaperto;\mathbb S^1)$ 
such that 
	\begin{equation*}
\begin{aligned}
&\int_{\sopraaperto}|\nabla \apprextuk|^2\dx+\Hausdorff
(S_{\apprextuk})+ k\int_{\sopraaperto}|\apprextuk-\extu|\dx
\\
=& \min_{w\in SBV^2(\sopraaperto;\mathbb S^1)}\left( 
\int_{\sopraaperto}|\nabla w|^2\dx+
\Hausdorff
(S_w)+ k\int_{\sopraaperto}|w-\extu|\dx
		\right)\,.
\end{aligned}
	\end{equation*}
Clearly 
$\|\apprextuk\|_\infty=1$ and 
\begin{equation}\label{est1}
 \int_{\sopraaperto}|\nabla \apprextuk|^2\dx+\Hausdorff(S_{\apprextuk})
+ k\int_{\sopraaperto}|\apprextuk-\extu|\dx
\le \int_{\sopraaperto}|\nabla \extu|^2\dx+\Hausdorff(S_\extu)<+\infty\,.
\end{equation}
Thus, in particular, $(\apprextuk)$ converges to $\extu$ in 
$L^1(\sopraaperto;\mathbb{S}^1)$ as $k\to+\infty$, and $(i)$ follows. 
By \cite{Ambrosio} we also get, up to a not relabelled subsequence,
\begin{equation*}
	\nabla \apprextuk\rightharpoonup\nabla \extu\quad \text{in 
}L^2(\sopraaperto;\R^{\dimension\times \dimension})\,,
\qquad \liminf_{k\to+\infty}\Hausdorff
(S_{\apprextuk})\ge \Hausdorff
(S_\extu)\,.
\end{equation*}
This and \eqref{est1} imply
\begin{equation*}
\begin{split}
\int_{\sopraaperto}|\nabla \extu|^2\dx+\Hausdorff
(S_\extu)= \lim_{k\to+\infty}
\left(\int_{\sopraaperto}|\nabla \extuk|^2\dx+\Hausdorff
(S_{\extuk})\right)
\,,\end{split}
\end{equation*}
from which we readily deduce $(ii)$ and $(iii)$. 
Again by \cite{Ambrosio} and \eqref{6} we have 
\begin{equation*}
\Hausdorff
(S_\extu\cap\Omega) \le \liminf_{k\to+\infty}
\Hausdorff
(S_{\apprextuk}\cap\Omega)\le \liminf_{k\to+\infty}
\Hausdorff
(S_{\apprextuk}\cap\overline\Omega)\,,
\end{equation*}
\begin{equation*}
\Hausdorff
(S_\extu\setminus\Omega)= \Hausdorff
(S_\extu\setminus\overline\Omega)
\le \liminf_{k\to+\infty}
\Hausdorff
(S_{\apprextuk}\setminus\overline\Omega)\,,
\end{equation*}
which in turn imply $(iv)$. Moreover,
invoking \cite[Lemma 4.5]{Carriero-Leaci} we have 
$\Hausdorff
(\overline S_{\apprextuk}\setminus S_{\apprextuk})=0$. 
Finally, thanks to the density estimate in \cite[Lemma 4.9]{Carriero-Leaci} 
and arguing exactly as in \cite[Proposition 5.3]{AT90} 
it can be shown that 
$$\Hausdorff
(S_{\apprextuk}\cap K)=\mathcal{M}^{\codone}(S_{\apprextuk}\cap K)\,,$$ 
for every compact set $K\subset\sopraaperto$ and $(v)$ is proven.
\end{proof}
A similar result holds also for liftings:
\begin{prop}[\textbf{Density result for liftings}]\label{prop:density2}
	Let  $\Omega$ and $\sopraaperto\subset\R^\dimension$ 
be as in Proposition~\ref{prop:density}. 
Let $u\in SBV^2(\Omega;\Suno)$ and $\varphi \in SBV^2(\Om)
\cap L^\infty(\Om)$ be a lifting of $u$.
	Then $\varphi$ has an extension $\widehat \varphi\in 
SBV^2(\sopraaperto)\cap L^\infty(\sopraaperto)$  satisfying $\Hausdorff
(S_{\widehat\varphi}\cap\partial\Omega)=0$. 
Moreover, there exists a sequence $(\apprextliftk)\subset SBV^2
(\sopraaperto)$ such that:
	\begin{enumerate}[label=$(\roman*)$]
		\item $\apprextliftk$ converges to $\extlift$ in $L^1(\sopraaperto)$;
		\item $\nabla \apprextliftk$ converges to 
$\nabla \extlift$ in $L^2(\sopraaperto;\R^\dimension)$;
		\item $\lim_{k\to\infty}\Hausdorff
(S_{\apprextliftk})=\Hausdorff
(S_{\widehat \varphi})$;
		\item $\lim_{k\to\infty}\Hausdorff
(S_{\apprextliftk}\cap\overline\Omega)=\Hausdorff
(S_{\widehat \varphi}\cap\Omega)=\Hausdorff
(S_\varphi)$;
\item $\Hausdorff(\overline S_{\apprextliftk}\setminus S_{
\apprextliftk
})=0$ and 
$\Hausdorff(S_{\apprextliftk}\cap K)=\mathcal{M}^{\codone}(S_{\apprextliftk
}\cap K)$ for every compact set $K\subset\sopraaperto$;
		\item  $\extu_k:=e^{i\apprextliftk}$ converges to $u$ 
in $L^1(\Omega;\mathbb S^1)$.
	\end{enumerate}
\end{prop}
Note that (ii)-(iii) guarantee the convergence of $\MS(
\apprextliftk,1)$ to 
$\MS(\extlift,1)$ in $\sopraaperto$.
\begin{proof}
From  \cite[Lemma 4.11]{Braides-Notes}  we can 
find $\widehat\varphi$ as in the statement and $(\apprextliftk)\subset SBV^2(\sopraaperto)$ that satisfy (i)--(v). This in turn implies also (vi).
\end{proof}
\black

\subsection{Proof of Theorem \ref{thm1}}\label{subsec:proof_1}

	\noindent
If $\var_k\searrow0$ and $((u_k,v_k))_{k\ge1}\subset 
L^1(\Omega;\mathbb{S}^1)\times L^1(\Omega, [0,1])$ is
a sequence 
converging to $(u,v)$ in $L^1(\Omega;\mathbb{S}^1)\times L^1(\Omega)$,
arguing exactly as in  \cite{Focardi}, one gets
\begin{equation*}
	\liminfk
\AToneepsk(u_k,v_k)\ge 
\MSSuno
(u,v).
\end{equation*}
Therefore, we only need to prove the upper bound. Let 
$\var_k\searrow0$  and  $u\in SBV^2(\Omega;\mathbb{S}^1)$. 
We have to find a sequence 
$((u_k,v_k))\subset\domATgrande$ 
converging to $(u,1)$ in $L^1(\Omega;\mathbb{S}^1)\times L^1(\Omega)$ 
for which
\begin{equation}\label{up_bound}
	\limsupk \AToneepsk(u_k,v_k) \le 
\MSSuno
(u,1)\,.
\end{equation} We note that, in general, we cannot take $((u_k,v_k))\subset\domATpiccolo$,  as it happens for the example  described in the introduction.

We fix a bounded open 
set $\sopraaperto\subset\subset\R^\dimension$, with
$\Omega\subset\subset\sopraaperto$. By Proposition \ref{prop:density} it suffices to show \eqref{up_bound} for $u\in SBV^2(\sopraaperto;\mathbb{S}^1)$ with
\begin{equation}\label{Su}
	\Hausdorff(S_u\cap\partial \Omega)=0\,,\quad
	\Hausdorff(\overline S_u\setminus S_u)=0\,,\quad
	\Hausdorff(S_u\cap K)=\mathcal{M}^\codone(S_u\cap K)\,,
\end{equation}
for every compact set $K\subset \sopraaperto$.

By Theorem \ref{thm:dav-ign} and Remark \ref{rem:dav-ign} there exists a 
lifting $\varphi\in SBV^2(\sopraaperto)$ of $u$.
For any $0<\rho\ll1$ define
\begin{equation*}
	(S_u)_\rho:=\{x\in \sopraaperto\colon {\rm dist}(x, S_u)<\rho\}.
\end{equation*}
Let $0<\xi_k$ with $\lim_{k\to+\infty}{\xi_k}/\var_k=0$ be such that
%
%
 $(S_u)_{\xi_k}$ has Lipschitz boundary.  
By Remark \ref{rem:smooth_approx} applied to $A=(S_u)_{{\xi}_k}$ and $\delta=\var_k$, 
there exists $\varphi_k\in C^\infty((S_u)_{{\xi}_k})$ such that 
\begin{equation}\label{eq:approx_phik}
	\int_{(S_u)_{\xi_k}}|\varphi-\varphi_k|\dx<\var_k\,, \quad
	\int_{(S_u)_{\xi_k}}|\nabla \varphi_k|\dx\le |D\varphi|\big((S_u)_{\xi_k}\big)+\var_k\,,
\end{equation}
and, in the sense of $BV$-traces,
\begin{equation}\label{eq:phiktraces}
\varphi_k=
\varphi\quad \text{ a.e. on }\quad \partial (S_u)_{\xi_k}\,.
\end{equation}
We define $\altrolifting_k\in SBV^2(\sopraaperto)$ as follows:
\begin{equation*}
	\altrolifting_k:=\begin{cases}
		\varphi_k&\text{ in }(S_u)_{\xi_k}\\
		\varphi & \text{ in }\sopraaperto\setminus (S_u)_{\xi_k}
\end{cases}\,.
\end{equation*}
Then by \eqref{eq:approx_phik}
the sequence $(\omega_k)$ 
converges to $\varphi$ strictly in $BV(\sopraaperto)$ and
%
	{
%
\begin{equation*}
S_{
\altrolifting_k
} \subseteq S_\varphi.
\end{equation*}
We set 
\begin{equation*}
	u_k:=e^{i\altrolifting_k}=\begin{cases}
		e^{i\varphi_k}&\text{ in }(S_u)_{\xi_k}\\
		u & \text{ in }\sopraaperto\setminus (S_u)_{\xi_k}
	\end{cases}
\end{equation*}
so that 
\begin{equation}\label{eq:prima_inclusione_importante}
u_k \in W^{1,1}(\sopraaperto;\mathbb{S}^1)
\end{equation}
and $(u_k)$ converges to $u$ in $L^1(\Omega;\mathbb{S}^1)$. Since we modified $u$ only in a neighbourhood of $S_u$ in general $u_k\notin W^{1,2}(\widehat\Omega;\mathbb S^1)$, as it happens in the example in the introduction.

The construction of $(v_k)$ is done as in \cite{Chambolle}.
 Precisely, we define $(v_k)\subset W^{1,2}(\sopraaperto)$ as 
\begin{equation*}
	v_k(x):=\begin{cases}
		0& x \in  (S_u)_{\xi_k}\\
		1-{\rm exp}\left(-\frac{\dist(x, S_u) -\xi_k}{2\var_k}
\right) & x \in  \sopraaperto\setminus(S_u)_{\xi_k}
	\end{cases}\,,
\end{equation*}

Then  $(v_k)$ converges to 1 in $L^1(\sopraaperto)$ as $k \to +\infty$. 
Moreover 
$$
v_k|\nabla u_k|\in L^2(\sopraaperto)
$$
 and hence, using \eqref{eq:prima_inclusione_importante},
we get the crucial inclusion 
$$
(u_k,v_k)\subset
		\domATgrande.
$$
It remains to show \eqref{up_bound}. 
To this aim we have
\begin{equation}\label{bulk}
	\limsupk	\int_{\Omega}v_k^2|\nabla u_k|^2\dx\le
\limsupk \int_{\Omega\setminus(S_u)_{\xi_k}}|\nabla u|^2\dx\le \int_\Omega|\nabla u|^2\dx\,.
\end{equation}
Moreover 
\begin{equation}\label{surface}
\ModicaMortolaepsk(v_k,\Om)		
= \frac{\vert (S_u)_{\xi_k}\cap\overline\Omega\vert}{\var_k}
+ 
\ModicaMortolaepsk(v_k,\Omega\setminus(S_u)_{\xi_k}).
\end{equation}
{}From \eqref{Su} it follows
\begin{equation}\label{eq:faci}
	\limsupk \frac{\vert
(S_u)_{\xi_k}\cap\overline\Omega\vert}{2\var_k}
= \limsupk \frac{\vert 
(S_u)_{\xi_k}\cap\overline\Omega\vert}{2\xi_k}\frac{\xi_k}{\var_k}= 0.
\end{equation}

Now, we estimate the second term on the right hand-side of \eqref{surface}.
At almost all points in $\Omega\setminus(S_u)_{\xi_k}$ 
there holds 
$$
|\nabla v_k|=\frac1{2\var_k} {\rm exp}\left(-\frac{
\dist(x, S_u) 
-\xi_k}{2\var_k}\right) |\nabla \dist(x, S_u) 
|=\frac1{2\var_k }{\rm exp}\left(-\frac{
\dist(x, S_u) 
-\xi_k}{2\var_k}\right)\,,$$
from which, using also the coarea formula we 
deduce
\begin{equation}\label{surface2}
	\begin{split}
\ModicaMortolaepsk\left(v_k,\Omega\setminus(S_u)_{\xi_k}\right)
&=\frac1{2\var_k} 	\int_{\Omega\setminus(S_u)_{\xi_k}}  {\rm exp}\left(-\frac{
\dist(x, S_u) 
-\xi_k}{\var_k}\right)\dx
		\\
		&=\frac1{2\var_k }\int_{\xi_k}^{+\infty}  
		e^{-\frac {t-\xi_k}{\var_k}} 
\Hausdorff(\partial E_t)\dt\,,
	\end{split} 
\end{equation}
with $E_t:=\{x\in\Omega\colon {\rm dist}(x,{S_u})>t\}$.  
Next, having that $$\Hausdorff(\partial E_t)=
\frac{{\rm d}}{\dt} \bigg(\int_0^t\Hausdorff(\partial E_s)\ds\bigg)=
\frac{{\rm d}}{\dt} \vert (S_u)_t\cap\overline\Omega\vert,
$$
integrating by parts first, and using the change of variable  $s=\frac{t-\xi_k}{\var_k}$ we obtain 
\begin{equation}\label{surface3}
	\begin{split}
& \frac1{2\var_k} \int_{\xi_k}^{+\infty}  
e^{-\frac {t-\xi_k}{\var_k}} \Hausdorff(\partial E_t)\dt
\\ 
=& 
		-\frac1{2\var_k}\vert (S_u)_{\xi_k}\vert
+
\frac{1}{2\var^2_k}\int_{\xi_k}^{+\infty} e^{-\frac {t-\xi_k}{\var_k}}\vert
(S_u)_t\cap\overline\Omega\vert\dt
\\
=& 
		-\frac1{2\var_k}\vert (S_u)_{\xi_k}
\vert+\int_{0}^{+\infty} (s+\var_k)e^{-s}\frac{ \vert (S_u)_{s\var_k
+\xi_k}\cap\overline\Omega\vert}{2(s\var_k+\xi_k)}\ds.
	\end{split}
\end{equation}
Gathering together \eqref{surface}--\eqref{surface3} we find
\begin{equation}\label{surface4}
\ModicaMortolaepsk(v_k,\Om)
=\frac1{2\var_k}\vert (S_u)_{\xi_k}\cap\overline\Omega\vert
+\int_{0}^{+\infty} (s+\var_k)e^{-s} \frac{ \vert
(S_u)_{s\var_k+\xi_k}\cap\overline\Omega\vert}{2(s\var_k+\xi_k)}\ds\,.
\end{equation}
{}From \eqref{Su} it follows
\begin{equation*}
	\limsupk \frac{ \vert
(S_u)_{s\var_k+\xi_k}\cap\overline\Omega\vert}{2(s\var_k+\xi_k)}= 
\Hausdorff(S_u\cap\overline\Omega)\,.
\end{equation*}
This, together with the fact that $\int_0^{+\infty}(s+\var_k)e^{-s}\ds\to1$ as $k\to+\infty$,
imply
\begin{equation}\label{surface5}
	\limsupk
\ModicaMortolaepsk(v_k,\Om)
\le \Hausdorff(S_u\cap\overline\Omega)\,.
\end{equation}
Eventually, combining \eqref{bulk}, \eqref{eq:faci} and \eqref{surface5} 
we infer \eqref{up_bound}.
\qed

\subsection{Proof of Theorem \ref{thm2}}\label{subsec:proof_2}

\noindent
\textit{Step 1: Lower bound}. Let $\var_k\searrow0$ as $k\to+\infty$.
We have to show that, for every sequence $((u_k,v_k))_{k\ge1}\subset L^1(\Omega;\mathbb{S}^1)\times L^1(\Omega)$ 
converging to $(u,v)$ in $L^1(\Omega;\mathbb{S}^1)\times L^1(\Omega)$,
\begin{equation}\label{lw_bound2}
	\liminfk
\ATtwoepsk(u_k,v_k)\ge 
\MSlift(u,v)\,.
\end{equation}
We may assume 
\begin{equation}\label{eq:assume_C}
\sup_{k \in \NN} 
\ATtwoepsk(u_k,v_k)\leq C <+\infty
\end{equation}
 so that $(u_k,v_k)\in\domATpiccolo$, $v=1$ a.e. in $\Omega$
and, 
up to a not relabelled subsequence, 
\begin{equation*}
	\liminfk
\ATtwoepsk
(u_k,v_k)=	\limk
\ATtwoepsk
(u_k,v_k),
\end{equation*}
\begin{equation*}
\liminfk 
\ModicaMortolaepsk(v_k, \Om)
= 
\limk
\ModicaMortolaepsk(v_k, \Om).
\end{equation*}
From Theorem~\ref{thm:Ambrosio-Tortorelli} 
and the decoupling property \eqref{eq:decoupling},
it follows that 
\begin{equation*}
	\begin{split}
&
	\liminf_{k\to+\infty}
\ATtwoepsk
(u_k,v_k)\ge \int_\Omega|\nabla u|^2\dx+\Hausdorff(S_u)\,,
\\
&	\liminf_{_k\to+\infty}\int_{\Omega}v_k^2|\nabla u_k|^2\dx\ge  
\int_\Omega|\nabla u|^2\dx\,,
\\	
&\liminfk 
\ModicaMortolaepsk(v_k, \Om)
\ge \Hausdorff(S_u)\,.
	\end{split}
\end{equation*}
In particular, from \eqref{eq:assume_C}, $u\in SBV^2(\Omega;\mathbb{S}^1)$.
Hence it remains to show 
\begin{equation}\label{liminf-MM}
	\limk
\ModicaMortolaepsk(v_k, \Om)
\ge \minvalueu2.
\end{equation}
For every $k\ge1$ we start by 
selecting a lifting
$\varphi_k\in W^{1,2}(\Omega)$ of $u_k$.
Since $|\nabla u_k|=|\nabla \varphi_k|$ we have
\begin{equation*}
	+\infty>C\ge \ATtwoepsk(u_k,v_k)=\int_\Omega\left(v_k^2|\nabla \varphi_k|^2+ \var_k|\nabla v_k|^2+\frac{(v_k-1)^2}{4\var_k}\right)\dx\,.
\end{equation*}
Thus, using the coarea formula,
\begin{equation}\label{eq:etaprimo}
C\geq \MMepsk(v_k, \Om)
\ge \int_\Omega (1-v_k)|\nabla v_k|\dx \\
= \int_0^1(1-t)\Hausdorff(\partial^* F_k^t)\dt
\end{equation}
for any $k \geq 1$, where $F_k^t:=\{x\in \Om:v_k(x)\leq t\}$. Let $\etaprimo,\etasecondo
\in(0,1)$, $\etaprimo<\etasecondo$ be fixed.
By \eqref{eq:etaprimo} and the mean value theorem 
there exists $\livellok
\in (\etaprimo,
\etasecondo)$ 
such that 
\begin{equation}\label{eq:etaeta}
C\ge (1-\etasecondo)
(\etasecondo-\etaprimo)
\Hausdorff(\partial^*F_k^{\livellok})\,.
\end{equation}
Moreover 
\begin{equation}\label{eq:due}
C\ge \int_\Omega v_k^2|\nabla\varphi_k|^2\dx\ge 
(\etaprimo)^2\int_{\Omega}
\chi_{\Omega\setminus F^{\livellok}_k}
|\nabla\varphi_k|^2
\dx\,.
\end{equation}
Then, setting 
$$
\nuovoliftingk
:=\varphi_k\chi_{\Omega\setminus F^{\livellok}_k}\in 
SBV^2(\Omega),
$$
 we have $S_{\nuovoliftingk}\subset \partial^*F_k^{\livellok}$  and,
for the absolutely continuous parts of the
gradients,  $\nabla\nuovoliftingk
=\nabla \varphi_k\chi_{\Omega\setminus F_k^{\livellok}}$ and so,
from \eqref{eq:due} and \eqref{eq:etaeta}, 
\begin{equation}\label{eq:pallino}
	\int_\Omega|\nabla \nuovoliftingk
|^2\dx+ \Hausdorff(S_{\nuovoliftingk
})\le C(\etaprimo, \etasecondo)
\end{equation}
for some $C(\etaprimo, \etasecondo)>0$ 
depending on $\etaprimo,\etasecondo$ and independent of $k$. Let also 
$$
\overline 
u_k:=e^{i\nuovoliftingk
}=u_k\chi_{\Omega\setminus F_k^{\livellok}}+(1,0)\chi_{F_k^{\livellok}},
$$
which, by 
\eqref{eq:etaeta}, are 
uniformly bounded in $BV(\Om; \Suno)$. Since 
\begin{equation}\label{eq:zero_in_measure}
\vert F_k^{\livellok}\vert \to 0,
\end{equation}
the sequence $(\overline u_k)$ weakly star converges to $u$ 
in $BV(\Omega;\mathbb{S}^1)$. Hence, using \eqref{eq:pallino},
we can apply Theorem \ref{teo:compactness} 
to the sequence
$(\nuovoliftingk
)_{k\ge1}$ (with $p=2$) and get, for a not-relabelled subsequence, a 
Caccioppoli
partition $(E_i)_{i\in \mathbb{N}}$ of $\Omega$, sequences $(d^{\pip}_k)_{k\ge1}\subset\mathbb{Z}$ 
for any integer $i\ge1$ and a lifting $\varphi_\infty\in GSBV^2(\Omega)$ 
of $u$, such that 
	\begin{equation*}
	\begin{split}
		&\lim_{k \to +\infty}(\nuovoliftingk
(x)-2\pi d^{\pip}_k)= \varphi_\infty(x)\qquad 
\forall i \in \NN, 
\text{ for a.e. }x\in E_{i}\,,\\
		&\lim_{k \to +\infty}|\nuovoliftingk
(x)-2\pi d^{\pip}_k|=+\infty\qquad \forall i \in \NN, 
\text{ for a.e. }x\in \Om\setminus E_i.
	\end{split}
\end{equation*}
Again, using
\eqref{eq:zero_in_measure},
the same holds for $\varphi_k$, i.e.,
	\begin{equation}\label{eq:conv-in-Ei}
	\begin{split}
		&\lim_{k \to +\infty}(\varphi_k(x)-2\pi d^{\pip}_k)= \varphi_\infty(x)
\qquad 
\forall i \in \NN, 
\text{ for a.e. }x\in E_{i}\,,\\
		&\lim_{k \to +\infty}|\varphi_k(x)-2\pi d^{\pip}_k|= +\infty\qquad 
\forall i \in \NN, 
\text{ for a.e. }x\in \Om\setminus E_{i}.
	\end{split}
\end{equation}
Then \eqref{liminf-MM} follows if we show  
 \begin{equation}\label{liminf-MM-bis}
 	\liminfk \ModicaMortola(v_k, \Om) 
\ge
\Hausdorff(S_{\varphi_\infty}),
 \end{equation}
since, being $\varphi_\infty$ a lifting of $u$, we have 
$$
\Hausdorff(S_{\varphi_\infty}) \geq m_2[u].
$$
 For any integer $K\ge1$ we consider the truncated function 
 \begin{equation}\label{eq:tru}
 	\trunkphi_k^{i,K}:=((\varphi_k-2\pi d_k^{\pip})\wedge K)\vee (-K)\in W^{1,2}(\Omega)\,,
 \end{equation}
and 
$$
\phi_k^{i,K}:= \trunkphi_k^{i,K}\chi_{\Omega\setminus F_k^{\livellok}}.
$$

Since $S_{\phi_k^{i,K}}\subset\partial^*F_k^{\livellok}$ and $\jump{\phi_k^{i,K}}\le K$, we have
\begin{equation*}\begin{split}
|D\phi_k^{i,K}|(\Omega)&= \int_\Omega|\nabla \nuovoliftingk
|\chi_{\{|\varphi_k-2\pi d_k^{\pip}|<K\}}\dx+ \int_{S_{\phi_k^{i,K}}}\jump{\phi_k^{i,K}}\,{\rm d}\mathcal{H}^{n-1} \\
&\le \int_\Omega|\nabla \nuovoliftingk
|\dx+ K\mathcal{H}^{n-1}(\partial^* F_k^{\livellok})\le C\,.
	\end{split}
\end{equation*}
Hence, up to a subsequence (depending on $K$), $(\phi_k^{i,K})$ converges 
to some $\phi^{i,K}_\infty$ in $L^1(\Omega)$. 
Moreover 
from \eqref{eq:conv-in-Ei} it holds $\phi^{i,K}_\infty
:= (\varphi_\infty\wedge K)\vee (-K)$ in $E_i$ and 
$\Omega\setminus E_i=F_i^+\cup F_i^-$ are such that $\phi^{i,K}_\infty=\pm K$ in $F_i^{\pm}$. 
As $|F_k^{\livellok}|\to 0$ it follows that $\varphi_k^{i,K}$ 
converges to $\phi_\infty^{i,K}$ in $L^1(\Omega)$.
Hence the sequence
 $((\varphi_k^{i,K},v_k))$ converges
to $(\phi_\infty^{i,K},1)$ in $L^1(\Omega)\times L^1(\Omega)$ and 
\begin{equation*}
	\begin{split}
\int_\sottoaperto\left(v_k^2|\nabla \varphi_k^{i,K}|^2+ \var_k|\nabla v_k|^2+\frac{(v_k-1)^2}{4\var_k}\right)\dx\le C\,,
	\end{split}
\end{equation*}
for any open set $\sottoaperto\subset\Omega$, for some 
$C>0$ independent of $k$. Thus, from  the decoupling property \eqref{eq:decoupling},
 \begin{equation}\label{conv-AT-decoupling}
 	\begin{split}
 		&	\liminfk \int_\sottoaperto
v_k^2|\nabla \varphi_k^{i,K}|^2\dx\ge \int_\sottoaperto|\nabla \phi_\infty^{i,K}|^2\dx\,,\\[1em]
 	&	\liminfk \ModicaMortola_{\var_k}(v_k,\sottoaperto)\ge
 	\Hausdorff(S_{\phi_\infty^{i,K}}\cap \sottoaperto).
 	\end{split}
 \end{equation}
In particular $\phi_\infty^{i,K}\in SBV^2(\Omega)$.
We fix an integer $N\geq 1$ and set,
for all $i=1,\dots,N$,
\begin{align*}
	\Sigma_i^K:=(S_{\phi_\infty^{i,K}}\cap\partial^*E_i)\cap \Om\,,\qquad  S^K_i:=S_{\phi^{i,K}_\infty}\setminus \partial^*E_i\,, 
\end{align*}
\begin{equation*}
	\Sigma^{N,K}:=\bigcup_{i=1}^N \Sigma_i^K\,,
	\qquad S^{N,K}:=\bigcup_{i=1}^N S_i^K\,.
\end{equation*}
{We now argue exactly as in the proof of Lemma \ref{lem:ricoprimenti}, so, given $\delta 
	\in (0,1)$, after defining the families of balls $\mathcal B$ and $\mathcal B_\indice$
 ($\indice
=1,\dots,2N$), we arrive at 
	\begin{align}\label{key66bis_MM}
		\sum_{B\in \mathcal B\cap \mathcal B_\indice
 }\Hausdorff(B\cap S_\indice^K)&
		\geq \Hausdorff(S_\indice^K)- \frac{\delta 
\Hausdorff
(S^{N,K}\cup \Sigma^{N,K})}{1-\delta}\,,
	\end{align}
	for all $\indice=1,\dots,N$, and 
	\begin{align}\label{key67bis_MM}
		\sum_{B\in \mathcal B\cap \mathcal B_{N+h} }
\Hausdorff(B\cap \Sigma_h^K)\geq \Hausdorff
(\Sigma_h^K)- \frac{\delta \Hausdorff
(S^{N,K}\cup \Sigma^{N,K})}{1-\delta}\,,
	\end{align}
	for all $h=1,\dots,N$.
	From \eqref{conv-AT-decoupling}  we know that, for all $\indice
=1,\dots,N$,
	\begin{align}\label{S_nbis_MM}
		\liminfk \ModicaMortola_{\var_k}(v_k,B)\ge 	
\Hausdorff(S_\indice^K\cap B)\,,
	\end{align}
for any $B\in \mathcal B\cap\mathcal B_j$,	and for all $h=1,\dots,N$
	\begin{align}\label{Sigma_nbis_MM}
		\liminfk \ModicaMortolaepsk(v_k,B)\ge 	
\Hausdorff(\Sigma_h^K\cap B) \,,
	\end{align}
	for any ball $B\in \mathcal B\cap\mathcal B_{N+h}$.
	So, summing \eqref{S_nbis_MM} over all $B\in \mathcal B\cap \mathcal B_\indice$ and \eqref{Sigma_nbis_MM} over all $B\in \mathcal B\cap \mathcal B_{N+h}$, and then over $\indice
,h$ respectively, we conclude again, using \eqref{key66bis_MM} and \eqref{key67bis_MM},
	\begin{equation*}
		\begin{split}
 \liminfk \ModicaMortola_{\var_k}(v_k,\Om)&=
\lim_{k\rightarrow +\infty}\ModicaMortola_{\var_k}(v_k,\Om)\\&\ge
\sum_{\indice=1}^N
\Hausdorff(S_\indice^K)+\sum_{h=1}^N
\Hausdorff(\Sigma_h^K) - \frac{2\delta N 
\Hausdorff(S^{N,K}\cup \Sigma^{N,K})}{1-\delta}
		\end{split}
	\end{equation*}
	and by the arbitrariness of $\delta>0$, 
	$$\limk\ModicaMortola_{\var_k}(v_k,\Om)\ge
\Hausdorff(S^{N,K}\cup\Sigma^{N,K})\,.$$  
	{The left hand-side in the above inequality is a limit, hence it does not depend on the subsequence, and so in particular on $K$ and $N$.}
	Letting first $N\rightarrow +\infty$ and then $K\to +\infty$ 
we finally get
	\begin{align*}\liminfk {\rm MM}_{\var_k}(v_k,\Omega)\geq \Hausdorff(S_{\varphi_\infty})\,,
	\end{align*}
and \eqref{liminf-MM} follows.}

\medskip

\noindent
\textit{Step 2: Upper bound}. Let  $\var_k\searrow0$ and 
 $u\in SBV^2(\Omega;\Suno)$. We 
have to find a sequence $((u_k,v_k))\subset
\domATpiccolo$ converging to $(u,1)$ in $L^1(\Omega;\Suno)\times L^1(\Omega)$ and
\begin{equation}\label{up_bound2}
	\limsupk 
\ATtwoepsk
(u_k,v_k) \le \MSlift
(u,1)\,.
\end{equation}
We notice that we cannot follow the strategy used in the proof
of the upper bound in Theorem \ref{thm1}, since that construction cannot
ensure  the functions $u_k$ to belong to $W^{1,2}(\Om;\mathbb{S}^1)$.

By Corollary \ref{cor:existence-minimizer} 
we can select a jump minimizing lifting 
$\varphi \in GSBV^2(\Omega)$ of $u$,  
$u=e^{i\varphi}$ and 
\begin{equation}\label{quasi-min}
	\Hausdorff(S_{\varphi})= \minvalueu2.
\end{equation}
The main difference with the construction in Theorem \ref{thm1} is that we will modify $u$ in a neighbourhood of $S_\varphi$ instead of $S_u$.
We divide the proof into two cases.

\textit{Case 1:} $\varphi \in L^\infty(\Omega)$. 
Let $\sopraaperto$, with $\Omega\subset\subset\sopraaperto\subset\subset\R^2$, be open. 
Since $\varphi \in SBV^2(\Omega)\cap  L^\infty(\Omega)$, by Proposition~\ref{prop:density2} we can assume, without loss of generality, 
that $\varphi\in SBV^2(\sopraaperto)$ with 
\begin{equation}\label{Svarphi}
	\Hausdorff(S_{\varphi}\cap\partial \Omega)=0\,,\quad
	\Hausdorff(\overline S_{\varphi}\setminus S_{\varphi})=0\,,\quad
	\Hausdorff(S_{\varphi}\cap K)=\mathcal{M}^\codone(S_{\varphi}\cap K)\,,
\end{equation}
for every compact set $K\subset \sopraaperto$.
In this way, arguing as in \cite{AT90,AT92} we can construct $(\varphi_k,v_k)\in W^{1,2}(\sopraaperto) \times W^{1,2}(\sopraaperto)$ 
such that $(\varphi_k,v_k)\to(\varphi,1)$ in $L^1(\sopraaperto) \times L^1(\sopraaperto)$ and 
\begin{equation}\label{at-varphi}
\limsupk {\rm  AT}_{\var_k}(\varphi_k,v_k)
\le {\rm MS}(\varphi,1)\,.
\end{equation}
Next we let $u_k:=e^{i\varphi_k}\in W^{1,2}(\sopraaperto;\Suno)$, so that $(u_k,v_k)\to(u,1)$ in $L^1(\sopraaperto;\Suno)\times L^1(\sopraaperto)$. Moreover, from \eqref{quasi-min}, \eqref{at-varphi}, $|\nabla u_k|=|\nabla\varphi_k|$ and $|\nabla u|=|\nabla\varphi|$
we get $${\rm AT}_{\var_k}(\varphi_k,v_k)=\ATtwoepsk
(u_k,v_k)\,,\quad {\rm MS}(\varphi,1)={\rm MS}_{\rm lift}(u,1)$$ and so
\begin{equation*}
	\limsupk \ATtwoepsk
(u_k,v_k) \le {\rm MS}_{\rm lift}(u,1)\,.
\end{equation*}

\textit{Case 2:} $\varphi \notin L^\infty(\Omega)$. 
For $N\in\mathbb{N}$, $N\ge1$ we let $\varphi^{\pNp}
:=\varphi\wedge N\vee(-N)\in L^\infty(\Omega)$. Then, as $N\to+\infty$, we have 
$$\varphi^{\pNp}\to\varphi\quad \text{ in } L^1(\Omega)\,,\quad 
\nabla \varphi^{\pNp}\to \nabla\varphi\quad \text{ in } L^2(\Omega;\R^{n})\,,$$
so that 
$$u^{\pNp}:=e^{i\varphi^N}\to e^{i\varphi}=u\quad \text{ in } L^1(\Omega;
\Suno)\,,\quad\nabla u^{\pNp}\to \nabla u\quad \text{ in } L^2(\Omega;\R^{n\times n})\,.$$
In addition,
\begin{equation*}
	\Hausdorff(S_{\varphi^{\pNp}})\le\Hausdorff(S_{\varphi})\,.
\end{equation*}
This in turn implies 
\begin{equation}\label{MS-varphiN}
	\begin{split}
	\limsup_{N\to+\infty} \left(\int_\Omega|\nabla \varphi^{\pNp}|^2\dx+ 
\Hausdorff(S_{\varphi^{\pNp}})\right)&\le \int_\Omega|\nabla \varphi|^2\dx+
\Hausdorff(S_{\varphi})
\\&\le \int_\Omega|\nabla u|^2\dx+
\minvalueu2.
	\end{split}
\end{equation}
For each $N\ge 1$ we can argue as in case 1 and find a sequence $(u_k^{\pNp},v_k^{\pNp})\to (u^{\pNp},1)$ in $L^1(\Omega;\Suno)\times L^1(\Omega)$ 
satisfying
\begin{equation}\label{at-varphiN}
	\limsupk \ATtwoepsk
(u_k^{\pNp},v_k^{\pNp})
	\le \int_\Omega|\nabla \varphi^{\pNp}|^2\dx+ \mathcal{H}^{n-1}(S_{\varphi}^{\pNp})\,.
\end{equation}
Combining \eqref{MS-varphiN} with \eqref{at-varphiN} we have 
\begin{equation*}
	\limsup_{N\to+\infty}	\limsupk \ATtwoepsk
(u_k^{\pNp},v_k^{\pNp})
	\le \int_\Omega|\nabla u|^2\dx+
	\minvalueu2.
\end{equation*}
Now we conclude using
 a diagonal argument. Namely, for any fixed $N\ge1$ there exists $k_N\ge1$ such that for $k\ge k_N$:
	\begin{itemize}
		\item $\|u_k^{\pNp}-u^{\pNp} \|_{L^1(\Omega;\mathbb S^1)}+\|v_k^{\pNp}-1 \|_{L^1(\Omega)}\le \frac1N$;
		\item $\ATtwoepsk(u_k^{\pNp},v_k^{\pNp})\le \displaystyle \int_\Omega|\nabla \varphi^{\pNp}|^2\dx+ \mathcal{H}^{n-1}(S_{\varphi}^{\pNp})+\frac1N$.
	\end{itemize}
Without loss of generality we may assume $k_{N+1}\ge k_N$. Now, for every $k\ge1$, we take $(u_k,v_k):=(u_k^{\pNp},v_k^{\pNp})$ if $k\in [k_N,k_{N+1})\cap\mathbb N$ and the proof is concluded.
\qed

\section*{Acknowledgements}
R. Marziani has received funding from the European Union's Horizon research and innovation program under the Marie Skłodowska-Curie agreement No 101150549. 
The research that led to the present paper was partially supported by
a grant from the GNAMPA group of INdAM. The first and third authors also acknowledge the partial financial support of the PRIN project 2022PJ9EFL "Geometric Measure Theory: Structure of Singular
Measures, Regularity Theory and Applications in the Calculus of Variations'', PNRR Italia Domani, funded
by the European Union via the program NextGenerationEU (Missione 4
Componente 2) CUP:E53D23005860006. Views
and opinions expressed are however those of the authors only and do not necessarily reflect those
of the European Union or the European Research Executive Agency. Neither the European Union
nor the granting authority can be held responsible for them.

\end{document}